\documentclass[12pt]{amsart}
\usepackage{amssymb}
\usepackage{fancybox}
\usepackage{mathrsfs}
\usepackage{amsrefs}
\usepackage{color}
\usepackage[all]{xy}
\usepackage{ulem}
\usepackage{comment}
\usepackage{tikz}
\usepackage[margin=3cm]{geometry}

\makeatletter\tagsleft@false\makeatother

\theoremstyle{plain}
\newtheorem{thm}{Theorem}[section]
\newtheorem{prop}[thm]{Proposition}
\newtheorem{lem}[thm]{Lemma}
\newtheorem{cor}[thm]{Corollary}
\theoremstyle{definition}
\newtheorem{defi}[thm]{Definition}
\newtheorem{rem}[thm]{Remark}
\newtheorem{ex}[thm]{Example}

\newcommand{\ZZ}{\mathbb{Z}}
\newcommand{\ZZp}{\mathbb{Z}_{\geq 0}}
\newcommand{\QQ}{\mathbb{Q}}
\newcommand{\RR}{\mathbb{R}}
\newcommand{\CC}{\mathbb{C}}

\newcommand{\LL}{\mathbb{L}}

\newcommand{\KK}{\mathbb{K}}

\newcommand{\nPoly}{\mathbb{C}[x_{1},\dots,x_{n}]}
\newcommand{\nLPoly}{\mathbb{C}[x^{\pm}_{1},\dots,x^{\pm}_{n}]}

\newcommand{\calF}{\mathcal{F}}
\newcommand{\calH}{\mathcal{H}}
\newcommand{\calG}{\mathcal{G}}
\newcommand{\calL}{\mathcal{L}}
\newcommand{\calM}{\mathcal{M}}
\newcommand{\calS}{\mathcal{S}}

\newcommand{\pcalH}{{}^p\mathcal{H}}

\newcommand{\CS}{(\CC^*)}
\newcommand{\ZZpn}{\mathbb{Z}_{\geq 0}^n}
\newcommand{\RRpn}{\mathbb{R}_{\geq 0}^n}
\DeclareMathOperator{\Sp}{\mathrm{sp}}

\newcommand{\hlr}{\lhook\joinrel\longrightarrow}
\newcommand{\lr}{\longrightarrow}
\newcommand{\simar}{\stackrel{\sim}{\longrightarrow}}

\newcommand{\DR}{\mathrm{R}}
\newcommand{\GR}{\mathrm{Gr}}

\newcommand{\zset}{V}

\newcommand{\RGVZ}[1]{\DR\Gamma_{V\setminus \{0\}}(#1)_{0}}

\newcommand{\Ntpo}{\mathrm{NP}(f)}
\newcommand{\Ntph}{\Gamma_{+}(f)}
\newcommand{\Ntbd}{\Gamma_{f}}

\newcommand{\Vol}{\mathrm{Vol}}
\newcommand{\ov}[1]{\overline{#1}}

\newcommand{\inpro}[2]{\langle{#1,#2}\rangle}

\newcommand{\sumcom}[1]{\colon\mbox{\tiny #1}}

\DeclareMathOperator{\SPEC}{\mathrm{Spec}}
\DeclareMathOperator{\Relint}{\mathrm{rel_{\cdot}int}}

\DeclareMathOperator{\Conv}{\mathrm{Conv}}
\DeclareMathOperator{\Aff}{\mathrm{Aff}}

\DeclareMathOperator{\Edge}{\mathrm{Edge}}

\DeclareMathOperator{\supp}{\mathrm{supp}}

\newcommand{\DBC}[1]{\mathrm{D^{b}_{c}}(#1)}
\newcommand{\DB}[1]{\mathrm{D^{b}}(#1)}
\newcommand{\Perv}[1]{\mathrm{Perv}(#1)}
\newcommand{\Mod}[1]{\mathrm{Mod}(#1)}

\newcommand{\DBMHM}[1]{\mathrm{D^{b}MHM}(#1)}

\newcommand{\Id}{\mathrm{Id}}
\DeclareMathOperator{\Int}{\mathrm{Int}}

\newcommand{\Ftsigma}{F(\tl{\sigma})}

\newcommand{\Dppp}{\DR\pi'_{*}}

\newcommand{\nearf}{\calG^{\bullet}}

\newcommand{\pcHW}[2]{\pcalH^{#1}\Dppp W_{#2}\nearf}

\newcommand{\tl}[1]{\widetilde{#1}}

\newcommand{\muhat}{\hat{\mu}}

\newcommand{\link}{\mathrm{lk}}

\newcommand{\nearll}{{}^p\psi_{f,\lambda,\overline{\lambda}}(\CC_{\CC^n}[n])}

\newcommand{\Precneq}{\hspace{0.5mm}\preceq\hspace{-4mm}\raisebox{-0.95mm}{\scalebox{1}[0.4]{/}}\hspace{2mm}}

\DeclareMathOperator{\Exp}{\mathrm{exp}}

\title[Milnor monodromies and MHS for non-isolated singularities]
{Milnor monodromies and mixed Hodge structures for non-isolated hypersurface singularities}

\author[Takahiro SAITO]{\large{Takahiro SAITO}}

\address{Institute of Mathematics, University of
Tsukuba, 1-1-1, Tennodai,
Tsukuba, Ibaraki, 305-8571, Japan.}
\email{takahiro@math.tsukuba.ac.jp}


\begin{document}

\maketitle

\begin{abstract}
We study the Milnor monodromies of non-isolated hypersurface singularities and
show that the reduced cohomology groups of the Milnor fibers are concentrated in the
middle degree for some eigenvalues of the monodromies. 
As an application of this result,
we give an explicit formula for some parts of their Jordan normal forms.
\end{abstract}

\tableofcontents

\section{Introduction}\label{introduction}
The Milnor monodromies of complex hypersurface singularities are important subjects in singularity theory.
For isolated hypersurface singular points, we have some algorithms or formulas to compute them.
However, for non-isolated singular points, 
there still remain some difficulties in computing them explicitly.
In this paper, we will show that
even for non-isolated singular points,
the generalized eigenspaces of the Milnor monodromies for ``good" eigenvalues
have some nice properties similar to those for isolated singular points.
By this result, we give an explicit formula for some parts of the Jordan normal forms of the Milnor monodromies.

Let $f(x)\in\nPoly$ be a polynomial of $n(\geq 2)$ variables with coefficients in $\CC$ such that $f(0)=0$ and $V:=f^{-1}(0)\subset \CC^n$ the hypersurface defined by it.
We denote by $F_{f,0}$ the Milnor fiber of $f$ at $0$ and by
\[\Phi_{j}\colon H^{j}(F_{f,0};\CC)\simar H^{j}(F_{f,0};\CC)\]
the $j$-th Milnor monodromy for $j\in\ZZ$.
If the origin $0\in V$ is an isolated singular point of $V$,
the Milnor fiber $F_{f,0}$ is homotopic to a bouquet of some $(n-1)$-spheres $S^{n-1}$ by a celebrated theorem of Milnor~\cite{MIL}.
This implies that the reduced cohomology groups $\tl{H}^{j}(F_{f,0};\CC)$ vanish except for $j=n-1$.
However, if $0\in V$ is a non-isolated singular point,
we can not expect to have such a concentration in general.
On the other hand,
for a polynomial $f$ non-degenerate at $0$,
Varchenko~\cite{Varchenko} described explicitly the monodromy zeta function
\[\zeta_{f,0}(t):=\prod_{j\in\ZZ}\det(\Id-t\Phi_{j})^{(-1)^{j}}\in\CC(t)\]
in terms of the Newton polyhedron $\Gamma_{+}(f)$.
If $0\in V$ is an isolated singular point of $V$,
the $(n-1)$-th Milnor monodromy $\Phi_{n-1}$ is the only non-trivial one.
Therefore, in this case, we obtain a formula for the characteristic polynomial of $\Phi_{n-1}$.
However, if $0\in V$ is a non-isolated singular point, Varchenko's formula 
does not tell us any explicit information about the characteristic polynomial of each Milnor monodromy $\Phi_{j}$.

For non-isolated singular points,
there is a similar difficulty also for the mixed Hodge structures of the cohomologies of the Milnor fibers.
Recall that each cohomology group $H^{j}(F_{f,0};\QQ)$ of $F_{f,0}$ is endowed with a mixed Hodge structure $(H^{j}(F_{f,0};\QQ),F^{\bullet}, W_{\bullet})$ defined by Steenbrink~\cite{STEVAN} in the case where $0\in V$ is an isolated singular point and
by Navarro~\cite{Navsur} and M. Saito~\cite{MHM} in the case where $0\in V$ is a non-isolated singular point.
For $j\in\ZZ$ and an eigenvalue $\lambda\in\CC$ of the Milnor monodromy $\Phi_{j}$,
we denote by 
\[H^{j}(F_{f,0};\CC)_{\lambda}\subset H^{j}(F_{f,0};\CC)\]
the generalized eigenspace of $\Phi_{j}$ for $\lambda$.
For $p,q\in\ZZ$, 
we denote by $h^{p,q}_{\lambda}(H^{j}(F_{f,0};\CC))$
the $(p,q)$-mixed Hodge number for the eigenvalue $\lambda$ of the $j$-th cohomology group of $F_{f,0}$ i.e. the dimension of $\GR^{p}_{F}\GR^{W}_{p+q}H^{j}(F_{f,0};\CC)_{\lambda}$.
For $\lambda\in\CC$, we define a polynomial $E_{\lambda}(F_{f,0};u,v)\in\ZZ[u,v]$ with coefficients in $\ZZ$ by
\[E_{\lambda}(F_{f,0};u,v):=\sum_{p,q\in\ZZ}\sum_{j\in\ZZ}(-1)^{j}h^{p,q}_{\lambda}(H^{j}(F_{f,0};\CC))u^pv^q.\]
By using a description of the Motivic Milnor fiber of $f$ at $0$ (see Theorem~\ref{motivicdesc}),
we can describe $E_{\lambda}(F_{f,0};u,v)$ explicitly in terms of certain polynomials defined by the Newton polyhedron $\Gamma_{+}(f)$ (see Corollary~\ref{epolycomp1}).
Moreover, if $0\in V$ is an isolated singular point of $V$,
as in the previous discussion about Varchenko's formula,
we can describe $h^{p,q}_{\lambda}(H^{n-1}(F_{f,0};\CC))$ by our formula.
In particular, we obtain an explicit formula for the Jordan normal form of $\Phi_{n-1}$ (see Matsui-Takeuchi~\cite{MT}).
On the other hand, if $0\in V$ is a non-isolated singular point,
the formula for $E_{\lambda}(F_{f,0};u,v)$ does not tell us any explicit information about each mixed Hodge number $h^{p,q}_{\lambda}(H^{j}(F_{f,0};\CC))$.

In this paper, 
we follow an idea of Takeuchi-Tibar~\cite{TakeTib} for monodromies at infinity
and overcome the above-mentioned difficulties by introducing a finite subset $R_{f}\subset \CC$ of ``bad'' eigenvalues of the Milnor monodromies (see Definition~\ref{Rf}) as follows.
\begin{thm}[see Theorem~\ref{main1}]\label{intromain1}
Assume that $f$ is non-degenerate at $0$.
Then, for any $\lambda\notin R_{f}$ we have a concentration:
\[\tl{H}^{j}(F_{f,0};\CC)_{\lambda}\simeq 0\quad ( j\neq n-1).\]
\end{thm}
Note that a more general but less explicit concentration result was given in \cite[Corollary~6.1.7]{DIMCA}. 
By this theorem and Varchenko's formula,
we can compute the multiplicities of eigenvalues $\lambda\notin R_{f}$ in $\Phi_{n-1}$ as follows.
\begin{cor}(see Corollary~\ref{multico})
In the situation of Theorem~\ref{intromain1},
for any $\lambda\notin R_{f}$ the multiplicity of the eigenvalue $\lambda$ in the Milnor monodromy $\Phi_{n-1}$ is equal to that of the factor $(1-\lambda t)$ in a rational function
\[\prod_{\emptyset\neq I\subset \{1,\dots,n\}}\prod_{i=1}^{k_{I}}(1-t^{d_{I,i}})^{(-1)^{n-|I|}\mathrm{Vol}_{\ZZ}(\Gamma_{I,i})}.\]
For the definitions of $k_{I}$, $\Gamma_{I,i}$, $d_{I,i}$ and $\mathrm{Vol}_{\ZZ}(\Gamma_{I,i})$, see Section~\ref{chapmil}.   
\end{cor}

Moreover, for such $\lambda$ we obtain the mixed Hodge numbers $h^{p,q}_{\lambda}(H^{n-1}(F_{f,0};\CC))$ by our formula for $E_{\lambda}(F_{f,0};u,v)$.

If $0\in V$ is an isolated singular point,
the filtration on $H^{n-1}(F_{f,0};\CC)_{\lambda}$ induced by the weight filtration coincides with the monodromy filtration of $\Phi_{n-1}$.
This implies that 
the Jordan normal form of $\Phi_{n-1}$ for an eigenvalue $\lambda$
can be described by the mixed Hodge numbers $h^{p,q}_{\lambda}(H^{n-1}(F_{f,0};\CC))$.
On the other hand, to the best of our knowledge,
if $0\in V$ is a non-isolated singular point of $V$,
the geometric meaning of the weight filtrations on $H^{j}(F_{f,0};\CC)_{\lambda}$ is not fully understood yet.
For this problem, we obtain the following.
\begin{thm}[see Theorem~\ref{main2}]\label{intromain2}
In the situation of Theorem~\ref{intromain1},
for $\lambda\notin R_{f}$
the filtration on $H^{n-1}(F_{f,0};\CC)_{\lambda}$ induced by the weight filtration
on $H^{n-1}(F_{f,0};\QQ)$ coincides with the monodromy filtration of $\Phi_{n-1}$ centered at $n-1$.
\end{thm}

We denote by $J_{k,\lambda}$ the number of the Jordan blocks with size $k$ for an eigenvalue $\lambda$ in the Jordan normal form of the Milnor monodromy $\Phi_{n-1}$.
Combining the above theorem with our description of $h^{p,q}_{\lambda}(H^{n-1}(F_{f,0};\CC))$,
for any eigenvalue $\lambda\notin R_{f}$ we can describe it by using the Newton polyhedron $\Ntph$ as follows.

\begin{cor}[see Corollary~\ref{jordan}]
In the situation of Theorems~\ref{intromain1} and \ref{intromain2}, 
for $\lambda\notin R_{f}$
we have
\begin{align*}
\sum_{0\leq k\leq n-1}J_{n-k,\lambda}u^{k+2}
=\sum_{{F\prec \Ntph\colon \mbox{\tiny admissible}}}
u^{\dim{F}+2}l^{*}_{\lambda}(\Delta_{F},\nu;1)\cdot\tl{l}_{P}(\calS_{\nu},\Delta_{F};u^2),
\end{align*}
where in the sum $\Sigma$ of the right hand side the face $F$ ranges through the admissible compact ones of $\Ntph$ (see Definition~\ref{admissible}).
For the definition of the polynomials $l^{*}_{\lambda}(\Delta_{F},\nu;u)$ and $\tl{l}_{P}(\calS_{\nu},\Delta_{F};u)$, see Section~\ref{KaSpoly}.
\end{cor}

We also apply our results to obtain a formula for the Hodge spectrum of the Milnor fiber $F_{f,0}$ (see Corollary~\ref{spfor}).

\subsection*{Acknowledgments}
The author would like to express his profound gratitude to Professor Kiyoshi Takeuchi
for the helpful suggestions and sincere encouragement during this work.
He would also like to express his sincere thanks to Professor Claude Sabbah for answering many questions and the helpful
comments and discussions, which led to an improvement of the proof of Theorem~\ref{main2}, present in Remark~\ref{newproof}.
He wishes to thank Yuichi Ike for useful discussions.
He also thanks the referee for useful comments.
This work was supported by JSPS KAKENHI Grant Number 17J00480.

\section{Milnor fibration}\label{chapmil} 
Let $f(x)\in\nPoly$ be a non-constant polynomial of $n$ variables with coefficients in $\CC$ such that $f(0)=0$.
For a natural number $m\geq 1$ and a positive real number $r>0$,
we denote by $B(0,r)$ the open ball in $\CC^m$ centered at the origin $0$ with radius $r$.
Set $B(0,r)^*:=B(0,r)\setminus \{0\}$.

\begin{thm}[Milnor~\cite{MIL}]\label{milfib}
Fix a sufficiently small $\epsilon>0$.
Then for a sufficiently small $(\epsilon \gg)\eta>0$ the restriction of $f\colon \CC^n\lr \CC$
\[f\colon B(0,\epsilon)\cap f^{-1}(B(0,\eta)^{*})\lr B(0,\eta)^*\]
is a locally trivial fibration.
Moreover, if the origin $0\in\CC^n$ is an isolated singular point of $V:=f^{-1}(0)\subset \CC^n$,
its fiber is homotopy equivalent to a bouquet (wedge sum) of some $(n-1)$-dimensional spheres.
\end{thm}

This locally trivial fibration does not depend on the choices of sufficiently small $\epsilon$ and $\delta$ up to diffeomorphism (see section 1.2 of \cite{OkNoCo}) and is called the \textit{Milnor fibration} of $f$ at $0$, and its general fiber $F_{f,0} $ is called the \textit{Milnor fiber} of $f$ at $0$.
By the above theorem, 
if $0$ is an isolated singular point of $V$,
it follows that the reduced cohomology group $\tl{H}^{j}(F_{f,0};\CC)$ vanishes for $j\neq n-1$.
We get an action of the fundamental group $\pi_{1}(B(0,\eta)^*)$ on $H^{j}(F_{f,0};\CC)$, and thus can define an automorphism
\[\Phi_{j}\colon H^{j}(F_{f,0};\CC)\simar H^{j}(F_{f,0};\CC)\]
for each $j\in\ZZ$.
We call it the \textit{$j$-th Milnor monodromy} of $f$ at $0$.
It is well-known that $\Phi_{j}$ is a quasi-unipotent linear operator for any $j\in\ZZ$.
For $\lambda\in\CC$ and $j\in\ZZ$ denote by $H^{j}(F_{f,0};\CC)_{\lambda}$ the generalized eigenspace of $\Phi_{j}$ for the eigenvalue $\lambda$.
Let $\Phi_{j,\lambda}$ be the restriction of $\Phi_{j}$ to $H^{j}(F_{f,0};\CC)_{\lambda}$.
To study the eigenvalues of the Milnor monodromies $\Phi_{j}$ and their multiplicities,
we introduce the following rational function.

\begin{defi}
In the situation as above, we define the \textit{monodromy zeta function $\zeta_{f,0}(t)\in\CC(t)$ of $f$ at $0$} by
\[\zeta_{f,0}(t):= \prod_{j\in\ZZ}\det(\Id- t\Phi_{j})^{(-1)^{j}}\in\CC(t),\]
where $\Id$ is the identity map of $H^{j}(F_{f,0};\CC)$ to itself.
\end{defi}

Since $\Phi_{j}$ are automorphisms,
the polynomials $\det(\Id- t\Phi_{j})$ determine the characteristic polynomials of $\Phi_{j}$.
To introduce a formula for the monodromy zeta functions of Varchenko~\cite{Varchenko},
we prepare some notions.

\begin{defi}\label{newtontpoly}
Let $f(x)=\sum_{\alpha\in\ZZ^n}a_{\alpha}x^{\alpha}\in\CC[x_{1}^{\pm},\dots,x_{n}^{\pm}]$ be a Laurent polynomial with coefficients in $\CC$.
Then the \textit{Newton polytope} $\Ntpo\subset \RR^n$ of $f$ is the convex hull of
the set $\supp(f):=\{\alpha\in\ZZ^n\ |\ a_{\alpha}\neq 0\}\subset \RR^n$ in $\RR^n$. 
\end{defi}

\begin{defi}\label{newtonph}
Let $f(x)=\sum_{\alpha\in\ZZpn}a_{\alpha}x^{\alpha}\in\nPoly$ be a polynomial with coefficients in $\CC$ such that $f(0)=0$.
\begin{enumerate}
\item The \textit{Newton polyhedron} $\Gamma_{+}(f)\subset \RR^n$ of $f$ at the origin $0\in\CC^n$ is the convex hull of $\bigcup_{\alpha\in\supp(f)}\{\alpha +\RRpn\}\subset \RR^n$ in $\RR^n$.
\item The \textit{Newton boundary} $\Gamma_{f}\subset \Gamma_{+}(f)$ of $f$ is 
the union of the compact faces of $\Ntph$.
\item We say that the polynomial $f$ is \textit{convenient} if $\Gamma_{+}(f)$ intersects the positive part of each coordinate axis of $\RR^n$.
\end{enumerate}
\end{defi}

For a Laurent polynomial $f(x)=\sum_{\alpha\in\ZZ^n}a_{\alpha}x^{\alpha}$
and a polytope $F$ in $\RR^n$,
we set $f_{F}(x):=\sum_{\alpha\in F}a_{\alpha}x^{\alpha}$.
\begin{defi}\label{nondeg}
Let $f(x)\in\CC[x_{1}^{\pm},\dots,x_{n}^{\pm}]$ be a Laurent polynomial with coefficients in $\CC$.
We say that $f$ is \textit{non-degenerate} if
for any face $F$ of $\Ntpo$
the hypersurface $\{x\in\CS^n\ |\ f_{F}(x)=0\}$ in $\CS^n$ is smooth and reduced.
\end{defi}

\begin{defi}\label{nondegatzero}
Let $f(x)\in\nPoly$ be a non-constant polynomial with coefficients in $\CC$ such that $f(0)=0$.
Then we say that $f$ is \textit{non-degenerate at $0$} if
for any compact face $F$ of $\Ntph$ the hypersurface $\{x\in\CS^n\ |\ f_{F}(x)=0\}$ in $\CS^n$ is smooth and reduced.
\end{defi}

Let $F$ be a lattice polytope in $\RR^n$ (i.e. its vertices are in $\ZZ^n$).
We denote by $\mathbf{L}_{F}$ (resp. $\Aff{F}$) the minimal linear (resp. affine) subspace of $\RR^n$ which contains $F$
and set $M_{F}:=\mathbf{L}_{F}\cap \ZZ^n$.
Note that $M_{F}$ is a lattice (i.e. a finite rank free $\ZZ$-module)
and we have $M_{F}\otimes_{\ZZ}\RR\simeq \mathbf{L}_{F}$.
Assume that $\dim{F}<n$ and $0\notin \Aff{F}$.
In this case there exists a unique primitive vector $v_{F}$ in the dual lattice $M_{F}^{*}$ of $M_{F}$ which takes a positive constant value on $\Aff{F}$.

\begin{defi}\label{latdist}
We define the \textit{lattice distance} $d_{F}\in\ZZ_{\geq 1}$ of $F$ from the origin $0\in\mathbf{L}_{F}$ to be the value of $v_{F}$ on $\Aff{F}$. 
\end{defi}

Let $f(x)\in\nPoly$ be a non-constant polynomial with coefficients in $\CC$ such that $f(0)=0$.
For a subset $I\subset \{1,\dots,n\}$, set 
\[
\RR^{I}:=\{(x_{1},\dots,x_{n})\in\RR^n\ |\ x_{i}=0\ (i\notin I)\}\simeq \RR^{|I|},\]
and let $\Gamma_{I,1},\dots,\Gamma_{I,k_{I}}$ be the $(|I|-1)$-dimensional compact faces of $\RR^{I}\cap \Ntph$.
For $1\leq i\leq k_{I}$, we define an integer $d_{I,i}\in\ZZ_{>0}$ to be the lattice distance of $\Gamma_{I,i}$ from the origin $0\in\RR^I$.
Let $\Vol_{\ZZ}(\Gamma_{I,i})\in\ZZ_{>0}$ be the $(|I|-1)$-dimensional normalized volume of $\Gamma_{I,i}$.
Then we have the following.

\begin{thm}[Varchenko~\cite{Varchenko}]\label{Varform}
Assume that $f$ is non-degenerate at $0$.
Then we have
\begin{align}\label{Varf}
\zeta_{f,0}(t)=\prod_{\emptyset\neq I\subset \{1,\dots,n\}}\prod_{i=1}^{k_{I}}(1-t^{d_{I,i}})^{(-1)^{|I|-1}\mathrm{Vol}_{\ZZ}(\Gamma_{I,i})}.
\end{align}
\end{thm}

The monodromy zeta function $\zeta_{f,0}(t)$ being an alternating product of 
the polynomials $\det(\Id- t\Phi_{j})$,
we can not compute the eigenvalues of each Milnor monodromy $\Phi_{j}$ and their multiplicities by Varchenko's formula in general.
Recall that if $0\in V$ is an isolated singular point of $V$
we have $H^{j}(F_{f,0};\CC)=0$ for $j\neq 0,n-1$ and
$\det(\Id- t\Phi_{0})=1-t$ (here we assumed $n\geq 2$).
In this case we thus obtain
\[\zeta_{f,0}(t)=(1-t)\cdot\det(\Id-t\Phi_{n-1})^{(-1)^{n-1}}\]
and can compute the eigenvalues of $\Phi_{n-1}$ and their multiplicities by Varchenko's formula.
Let us recall a well-known condition for the origin $0\in V$ to be an isolated singular point.
\begin{prop}
If $f$ is convenient and non-degenerate at $0$, then the origin $0$ is a smooth or an isolated singular point of $V$.
\end{prop}

Therefore, if $f$ is convenient and non-degenerate at $0$,
we can describe the eigenvalues of the Milnor monodromy $\Phi_{n-1}$ and their multiplicities by the Newton polyhedron $\Ntph$.
In Section~\ref{chapmain}, 
we will show that even if $f$ is not convenient, we can compute the multiplicities of some ``good'' eigenvalues of $\Phi_{n-1}$ by Varchenko's formula (see Corollary~\ref{multico}).

Recall that the cohomology group $H^{j}(F_{f,0};\QQ)$ of the Milnor fiber $F_{f,0}$ is endowed with a mixed Hodge structure $(H^{j}(F_{f,0};\QQ), F^{\bullet}, W_{\bullet})$ defined by Steenbrink~\cite{STEVAN} in the case where $0\in V$ is isolated singular point, by Navarro~\cite{Navsur} and M.Saito~\cite{MHM} in the case where $0\in V$ is a non-isolated singular point.
To introduce a property of the weight filtration $W_{\bullet}$, we recall the following notion.
\begin{defi}
Let $r\in\ZZ_{\geq 1}$ and $N$ be a nilpotent endomorphism of a finite dimensional $\CC$-vector space $H$ such that $N^{r+1}=0$.
Then, there exists an increasing filtration $\{W_{k}\}_{k\in\ZZ}$ on $H$, which is uniquely determined by the following conditions:
\begin{enumerate}
\item $W_{-1}=\{0\}$ and $W_{2r}=H$,\\
\item $N(W_{k})\subset W_{k-2}$ \quad and\\
\item $N^{k}\colon W_{r+k}/W_{r+k-1}\simar W_{r-k}/W_{r-k-1}$\mbox{\ \ for $0\leq k\leq r$}.
\end{enumerate}
We call it the \textit{monodromy filtration of $N$ centered at $r$}.
\end{defi}
Note that the number of Jordan blocks (for the eigenvalue $0$) with size $k\in \ZZ_{\geq 1}$ in the Jordan normal form of $N$ is equal to
\[\dim{W_{r+1-k}/W_{r-k}}-\dim{W_{r-1-k}/W_{r-2-k}}.\]
We decompose the $(n-1)$-th Milnor monodromy $\Phi_{n-1}$ into the semisimple part $\Phi_{n-1}^{s}$ and the unipotent part $\Phi_{n-1}^{u}$ as $\Phi_{n-1}=\Phi_{n-1}^{s}\Phi_{n-1}^{u}$.
If $0\in V$ is an isolated singular point,
the filtration on $H^{n-1}(F_{f,0};\CC)_{\lambda}$ induced by the weight filtration is the monodromy filtration of the logarithm $\log{\Phi_{n-1}^{u}}$ of $\Phi_{n-1}^{u}$ centered at $n-1$ (resp. $n$) for $\lambda\neq 1$ (resp. $\lambda=1$) (see \cite{STEVAN}).
Therefore, we can recover the Jordan normal form of $\Phi_{n-1}$ for the eigenvalue $\lambda$ from the dimensions of the graded pieces $\GR^{W}_{k}H^{n-1}(F_{f,0};\CC)_{\lambda}$.
On the other hand, if $0\in V$ is a non-isolated singular point,
we can not expect such a relationship between the weight filtration $W_{\bullet}$ and the Milnor monodromy.
In Section~\ref{chapmain}, we will show that even if $f$ is not convenient (so $0\in V$ may be a non-isolated singular point) for a ``good'' eigenvalue $\lambda$ the filtration on $H^{n-1}(F_{f,0};\CC)_{\lambda}$ induced by $W_{\bullet}$ coincides with the monodromy filtration (see Theorem~\ref{main2}).

Finally, we describe the cohomology groups of the Milnor fibers and their mixed Hodge structures in terms of the nearby cycle functors.
For details, see~\cite{DIMCA}, \cite{HTT}, \cite{KS} and \cite{MHM}. 
For a field $\KK$,
we denote by $\KK_{\CC^n}$ the constant sheaf on $\CC^n$ with stalk $\KK$ and by $\psi_{f}(\KK_{\CC^n})$ the nearby cycle sheaf of $f$.
Recall that $\psi_{f}(\KK_{\CC^n})$ is an object of the derived category of constructible sheaves $\DBC{V}$ on $V=f^{-1}(0)$.
Then, for any $j\in\ZZ$
there exists an isomorphism 
\[H^{j}(\psi_{f}(\KK_{\CC^n})_{0})\simeq H^{j}(F_{f,0};\CC)\]
(see e.g. Proposition~4.2.2 of \cite{DIMCA}).
Moreover, there exists an automorphism of the complex $\psi_{f}(\KK_{\CC^n})$, called the monodromy automorphism.
The automorphism on $H^{j}(\psi_{f}(\KK_{\CC^n})_{0})\simeq H^{j}(F_{f,0};\KK)$ induced by it coincides with the Milnor monodromy $\Phi_{j}$.
For a complex number $\lambda\in\CC$, we denote by $\psi_{f,\lambda}(\CC_{\CC^n})$ the $\lambda$-part of $\psi_{f}(\CC_{\CC^n})$ with respect to the monodromy automorphism.
Then the $j$-th cohomology $H^{j}(\psi_{f,\lambda}(\CC_{\CC^n})_{0})$ is isomorphic to the generalized eigenspace $H^{j}(F_{f,0};\CC)_{\lambda}$ of $\Phi_{j}$ for the eigenvalue $\lambda$.
Recall that the complexes
\[\KK_{\CC^n}[n]\in\DBC{\CC^n}\]
and
\[{}^p\psi_{f}(\KK_{\CC^n}[n]):=\psi_{f}(\KK_{\CC^n}[n-1])\in\DBC{V}\]
are objects of the Abelian categories of perverse sheaves $\Perv{\CC^n}\subset{\DBC{\CC^n}}$ and $\Perv{V}\subset \DBC{V}$ respectively.
Moreover, $\QQ_{\CC^n}[n]$ and ${}^p\psi_{f}(\QQ_{\CC^n}[n])$ are the underlying perverse sheaves of the mixed Hodge modules $\QQ^H_{\CC^n}[n]$ and
$\psi_{f}^{H}(\QQ^{H}_{\CC^n}[n])$ respectively.
Let $j_{0}^{*}\psi_{f}^{H}(\QQ^{H}_{\CC^n})$ be the pull-back of the mixed Hodge module $\psi_{f}^{H}(\QQ^{H}_{\CC^n})\in \DBMHM{V}$ by the inclusion $j_{0}\colon \{0\}\hlr V$.
This is an object of the derived category of mixed Hodge modules $\DBMHM{\{0\}}$ (which is equivalent to the derived category of polarizable mixed Hodge structures),
and its underlying complex is ${}^p\psi_{f}(\QQ_{\CC^n}[n])_{0}(=j_{0}^{-1}({}^p\psi_{f}(\QQ_{\CC^n}[n])))$.
This implies that the cohomology groups of ${}^p\psi_{f}(\QQ_{\CC^n}[n])_{0}$ have mixed Hodge structures.
Thus we can endow $H^{j}(F_{f,0};\QQ)$ with a mixed Hodge structure for each $j\in\ZZ$.
Note that in the case where $0\in V$ is an isolated singular point,
the mixed Hodge structure of $H^{n-1}(F_{f,0};\QQ)$ was defined by Steenbrink~\cite{STEVAN} more elementarily.

\section{Motivic Milnor fibers}\label{chapmot}
\subsection{Motivic Milnor fibers}\label{subsecmot}
Let $\mu_{m}=\{x\in\CC\ |\ x^m=1\}$ be the cyclic group of order $m\in\ZZ_{\geq 1}$.
Let $\muhat=\varprojlim_{m}\mu_{m}$ be the projective limit with respect to the morphisms $\mu_{md}\lr \mu_{m}\ (x\mapsto x^{d})$.
Let $K_{0}^{\muhat}(\mathrm{Var_{\CC}})$ be the Abelian group generated by the symbols ${[X\circlearrowleft\muhat]}$ for algebraic varieties $X$ over $\CC$ with a good $\muhat$-action (i.e. a $\muhat$-action induced by a good $\mu_{m}$-action for some $m\in\ZZ_{\geq 1}$), and divided by some relations (see \cite{DLarc}).
For $[X\circlearrowleft\muhat]$ and $[X'\circlearrowleft\muhat]$ in $K_{0}^{\muhat}(\mathrm{Var_{\CC}})$
we can endow the product $X\times X'$ with a good $\muhat$ action and
define their multiplication $[X\circlearrowleft\muhat]\cdot[X'\circlearrowleft\muhat]$ by
$[X\times X' \circlearrowleft \muhat]$.
In this way, we can endow $K_{0}^{\muhat}(\mathrm{Var_{\CC}})$ with a ring structure and call it the \textit{monodromic Grothendieck ring}.
Set $\LL:=[\CC\circlearrowleft\muhat]\in K_{0}^{\muhat}(\mathrm{Var_{\CC}})$, where $\CC\circlearrowleft\muhat$ is the affine line with the trivial $\muhat$-action.
We denote by $\calM^{\muhat}_{\CC}$ the localization of the ring $K_{0}^{\muhat}(\mathrm{Var_{\CC}})$ obtained by inverting $\LL$.
For a non-constant polynomial $f(x)\in\nPoly$ such that $f(0)=0$,
Denef-Loeser~\cite{DLarc} defined the \textit{motivic Milnor fiber} $\calS_{f,0}$ of $f$ as an object in $\calM_{\CC}^{\muhat}$ by using the theory of arc spaces.
It is an ``incarnation of the Milnor fiber of $f$ at $0$''
in the ring $\calM^{\muhat}_{\CC}$ as we shall see in Theorem~\ref{HOREAL} below.
Let $X$ be an algebraic variety with a good $\mu_{m}$-action for some $m\in\ZZ_{\geq 1}$.
The generator of $\mu_{m}$ defines an automorphism $l\colon X\simar X$ of $X$ such that $l^m$ is the identity map.
Each cohomology group $H^{j}_{c}(X;\QQ)$ with compact support is endowed with Deligne's mixed Hodge structure $(H^{j}_{c}(X;\QQ), F^{\bullet},W_{\bullet})$ with an automorphism 
\begin{align}\label{actionl}
(l^*)^{-1}\colon H^{j}_{c}(X;\QQ)\simar H^{j}_{c}(X;\QQ).
\end{align}
For $\lambda\in\CC$, we denote by 
\[H^{j}_{c}(X;\CC)_{\lambda}\subset H^j_{c}(X;\CC)\]
the generalized eigenspace of the automorphism for the eigenvalue $\lambda$.
Moreover, for $p,q\in\ZZ$ and $\lambda\in\CC$
we define $h^{p,q}_{\lambda}(H^{j}_{c}(X;\CC))\in\ZZ_{\geq 0}$ to be the dimension of
$\GR^{p}_{F}\GR^{W}_{p+q}H^{j}_{c}(X;\CC)_{\lambda}$.
Then for $\lambda\in\CC$ we can define a ring homomorphism $E_{\lambda}(\ \cdot\ ;u,v)$ of $\calM^{\muhat}_{\CC}$ to the polynomial ring $\ZZ[u,v]$
which sends $[X\circlearrowleft\muhat]\in\calM_{\CC}^{\muhat}$ to
the polynomial
\[E_{\lambda}([X\circlearrowleft\muhat];u,v):=\sum_{p,q\in\ZZ}\sum_{j\in\ZZ}(-1)^{j}h^{p,q}_{\lambda}(H^{j}_{c}(X;\CC))u^pv^q\in\ZZ[u,v].\]
For an element $\sum_{i}[X_{i}\circlearrowleft\muhat]\in\calM_{\CC}^{\muhat}$ and $\lambda\in\CC$
we call the polynomial $\sum_{i}E_{\lambda}([X_{i}\circlearrowleft\muhat];u,v)$ the \textit{Hodge realization for $\lambda$} of $\sum_{i}[X_{i}\circlearrowleft\muhat]$. 
For $\lambda\in\CC$ we denote by $E_{\lambda}(F_{f,0};u,v)$ the \textit{equivariant Hodge-Deligne polynomial for the eigenvalue $\lambda$} of the mixed Hodge structures of the cohomology groups of the Milnor fiber $F_{f,0}$ with the automorphisms $\Phi_{j}$, i.e.
\[E_{\lambda}(F_{f,0};u,v):=\sum_{p,q\in\ZZ}\sum_{j\in\ZZ}(-1)^{j}h^{p,q}_{\lambda}(H^{j}(F_{f,0};\CC))u^pv^q\in\ZZ[u,v],\]
where $h^{p,q}_{\lambda}(H^{j}(F_{f,0};\CC))$ is the dimension of 
$\GR^{p}_{F}\GR^{W}_{p+q}H^{j}(F_{f,0};\CC)_{\lambda}$.
Then we have the following theorem of Denef-Loeser~\cite{DLarc}.
\begin{thm}[Denef-Loeser~\cite{DLarc}]\label{HOREAL}
For any $\lambda\in\CC$ we have
\[E_{\lambda}(\calS_{f,0};u,v)=E_{\lambda}(F_{f,0};u,v).\]
\end{thm}

Originally, $\calS_{f,0}$ is defined abstractly in \cite{DLarc} by using the theory of arc spaces.
However, by using a log resolution of the pair $(\CC^n, f^{-1}(0))$,
we can describe $\calS_{f,0}$ explicitly as follows.
Let $Y$ be a smooth algebraic variety over $\CC$ and $\pi\colon Y\lr \CC^n$ be a proper morphism such that $\pi^{-1}(V)$ is a normal crossing divisor of $Y$ and $\pi$ induces an isomorphism $Y\setminus \pi^{-1}(V)\simar \CC^n \setminus V$.
Let
\[\pi^{-1}(V)=E_{1}\cup\dots\cup E_{m}\]
be the irreducible decomposition of $\pi^{-1}(V)$.
We denote by $m_{i}$ the order of zeros along $E_{i}$ of $f\circ \pi$.
For a subset $I\subset \{1,\dots, m\}$,
we define 
\[E_{I}:=\cap_{i\in I}E_{i}, \quad E_{I}^{\circ}:=E_{I}\setminus \cup_{i\notin I}E_{i}\] and $m_{I}:=\gcd_{i\in I}(m_{i})$.
Moreover, we define a covering $\tl{E_{I}^{\circ}}$ of $E_{I}^{\circ}$ in the following way.
For a point in $E_{I}^{\circ}$, we take a Zariski open neighborhood $U$ of it in $Y$
on which for any $i\in I$ there exists a regular function $h_{i}$ such that $E_{i}\cap U=\{h_{i}=0\}$.
We have $f\circ \pi=f_{1}f_{2}^{m_{I}}$ on $U$, where we set $f_{1}=(f\circ \pi)\prod_{i\in I}h_{i}^{-m_{i}}$ and $f_{2}=\prod_{i\in I}h_{i}^{m_{i}/m_{I}}$.
Note that $f_{1}$ is a unit on $U$.
Then, we have a covering of $E^{\circ}_{I}\cap U$ defined by
\begin{align}\label{ecapu}
\{(z,y)\in \CC\times (E^{\circ}_{I}\cap U)\ |\ z^{m_{I}}=f_{1}^{-1}(y) \}.
\end{align}
Consider an open covering of $E_{I}^{\circ}$ by such affine open sets $E_{I}^{\circ}\cap U$.
Then by gluing together the varieties (\ref{ecapu}) in an obvious way, we obtain an $m_{I}$-fold covering $\tl{E^{\circ}_{I}}$ of $E_{I}^{\circ}$.
Moreover by the multiplication of $\exp(2\pi\sqrt{-1}/m_{I})\in\CC$ to the $z$-coordinate of (\ref{ecapu}),
we can endow $\tl{E_{I}^{\circ}}$ with a $\mu_{m_{I}}$-action (and also a $\muhat$-action induced by it).
We denote by $\tl{E_{I,0}^{\circ}}$ the base change of $\tl{E_{I}^{\circ}}\lr E_{I}^{\circ}$ by $\pi^{-1}(0)\cap E_{I}^{\circ}\hlr E_{I}^{\circ}$.
Then we obtain the following expression of $\calS_{f,0}$.
\begin{prop}[Denef-Loeser~\cite{DLarc}]\label{fmot}
In the situation as above, we have
\begin{align}\label{formot}
\calS_{f,0}= \sum_{\emptyset\neq I\subset \{1,\dots,m\}}(1-\LL)^{|I|-1}[\tl{E^{\circ}_{I,0}}] 
\end{align}
in $\calM_{\CC}^{\muhat}$.
\end{prop}

We denote by $\Sigma_{0}$ the dual fan of $\Ntph$ in $\RR^n$.
Let $\Sigma$ be a smooth subdivision of $\Sigma_{0}$,
and denote by $X_{\Sigma}$ the toric variety associated with it.
We denote by $\Sigma_{1}$ the fan which consists of all the faces of $\RR^n_
{\geq 0}$. 
Note that the toric variety associated with it is $\CC^n$.
Then, the morphism of fans $\Sigma\lr \Sigma_{1}$ induces a
morphism of toric varieties 
\[\pi \colon X_{\Sigma}\lr \CC^n.\]
If $f$ is non-degenerate at $0$,
we can take $\pi \colon X_{\Sigma}\lr \CC^n$ as a log resolution of $(\CC^n, V)$ (In fact, $\pi^{-1}(V)$ is normal crossing only in a neighborhood of $\pi^{-1}(0)$ in $X_{\Sigma}$.
Nevertheless, we can apply Proposition~\ref{fmot}.).
Moreover by calculating the right hand side of (\ref{formot}),
we can describe the motivic Milnor fiber $S_{f,0}$ explicitly in terms of the Newton boundary $\Ntbd$ as follows (see Matsui-Takeuchi~\cite[Section~4]{MT}). 
Assume that $f(x)=\sum_{\alpha\in\ZZpn}a_{\alpha}x^{\alpha}$ ($a_{\alpha}\in\CC$) is non-degenerate at $0$.
For a compact face $F$ of $\Ntph$, we
define a lattice polytope $\Delta_{F}\subset \RR^n$ by $\Delta_{F}:=\Conv({F\cup \{0\}})$.
Moreover, we define a polynomial $\tl{f_{\Delta_{F}}}$ to be $f_{\Delta_{F}}-1$.
Then we define a hypersurface $Z^{\circ}_{F}$ in $\SPEC{\CC[\ZZ^n\cap\Aff{F}]}\simeq \CS^{\dim{F}}$
by
\[Z^{\circ}_{F}=\{x\in \SPEC{\CC[\ZZ^n\cap\Aff{F}]}\ |\ f_{F}(x)=0\}\subset \CS^{\dim{F}},\]
and a hypersurface $Z^{\circ}_{\Delta_{F}}$ in $\SPEC{\CC[\ZZ^n\cap \Aff{\Delta_{F}}]}\simeq \CS^{\dim{F}+1}$ by
\[Z^{\circ}_{\Delta_{F}}=\{x\in \SPEC{\CC[\ZZ^n\cap\Aff{\Delta_{F}}]}\ |\ \tl{f_{\Delta_{F}}}(x)=0\}\subset \CS^{\dim{F}+1}.\]
We endow $Z_{F}^{\circ}$ with the trivial $\muhat$-action,
and $Z_{\Delta_{F}}^{\circ}$ with a good $\muhat$-action in the following way.
Let $\nu_{F}$ be the linear function on $\Aff{\Delta_{F}}\simeq \RR^{\dim{F}+1}$ which takes the value
$1$ on $F$
and $e_{F}\in \CS^{\dim{F}+1}\simeq \SPEC{\CC[\ZZ^n\cap\Aff{\Delta_{F}}]}\simeq \mathrm{Hom}_{\mathrm{group}}(\ZZ^n\cap\Aff(\Delta_{F}),\CC^*)$
be the element which corresponds to the group homomorphism
\[\Exp\left(2\sqrt{-1}{\nu_{F}(\cdot)}\right) \in \mathrm{Hom}_{\mathrm{group}}(\ZZ^n\cap\Aff(\Delta_{F}),\CC^*).\] 
Then $Z^{\circ}_{\Delta_{F}}$ is invariant by the multiplication by $e_{F}$
and hence we can endow $Z^{\circ}_{\Delta_{F}}$ with a $\mu_{d_{F}}$-action.
We thus obtain the elements $[Z^{\circ}_{{F}}\circlearrowleft\muhat]$ and $[Z^{\circ}_{\Delta_{F}}\circlearrowleft\muhat]$ in $\calM^{\muhat}_{\CC}$ for any compact face $F\prec \Ntph$.
For a compact face $F$ of $\Ntph$ we define also a subset $I_{F}\subset \{1,\dots,n\}$
to be the minimal one such that $F\subset \RR^{I_{F}}$,
and set $s_{F}=|I_{F}|$.
Then we have the following theorem.
\begin{thm}[see~Matsui-Takeuchi~\cite{MT}]\label{motivicdesc}
Assume that $f$ is non-degenerate at $0$.
Then we have
\[\calS_{f,0}=\sum_{F\prec \Ntph\colon \mbox{\tiny compact}}(1-\LL)^{s_{F}-\dim{F}-1}\Bigl\{(1-\LL)\cdot[Z^{\circ}_{{F}}\circlearrowleft\muhat]+[Z^{\circ}_{\Delta_{F}}\circlearrowleft\muhat]\Bigr\}\in \calM_{\CC}^{\muhat},\]
where in the sum $\Sigma$ the face $F(\neq \emptyset)$ ranges through the compact ones of $\Ntph$.   
\end{thm}
This theorem was proved only in the case where $f$ is convenient in \cite{MT}.
However, we can show it even if $f$ is not convenient similarly.

\subsection{Katz and Stapledon's polynomials}\label{KaSpoly}
In this subsection, we introduce some polynomials defined by Katz-Stapledon~ \cite{KaSthpoly}, \cite{KaSttrop} and Stapledon~\cite{SFM}.
Let $P\subset \RR^n$ be a polytope in $\RR^n$.
If a subset $F\subset  P$ is a face of $P$ (possibly $F=P$ or $F= \emptyset$), we write $F\prec P$.
For a pair of faces $F\prec F'$ of $P$,
we define $[F,F']$ to be the face poset $\{F''\prec P\ |\ F\prec F''\prec F'\}$,
and $[F,F']^*$ to be its opposite poset.

\begin{defi}
Let $F\prec F'$ be a pair of faces of $P$.
We define one-variable polynomials $g([F,F'];t)$ and $g([F,F']^*;t)$ with coefficients in $\ZZ$ of degree strictly less than ${(\dim F'-\dim F)}/2$ by the following inductive way.
If $F=F'$, we set $g([F,F'];t)=1$ and $g([F,F']^*;t)=1$.
If $F\Precneq F'$, we define them by
\begin{align*}
t^{\dim{F'}-\dim{F}}g([F,F'];t^{-1})=\sum_{F''\in [F,F']}(t-1)^{\dim{F'}-\dim{F''}}g([F,F''];t)\quad \mbox{and}
\\
t^{\dim{F'}-\dim{F}}g([F,F']^*;t^{-1})=\sum_{F''\in [F,F']^*}(t-1)^{\dim{F''}-\dim{F}}g([F'',F']^*;t).
\end{align*}
\end{defi}

Let $P$ be a lattice polytope in $\RR^n$ (i.e. all the vertices of $P$ are in $\ZZ^n$) and $\nu$ a convex $\RR$-valued function on $P$ 
which is piecewise $\QQ$-affine with respect to a lattice polyhedral subdivision of $P$.
Let $\calS_{\nu}$ be the coarsest such polyhedral subdivision.
A (possibly empty) polytope $F\subset P$ in $\calS_{\nu}$ is called a \textit{cell}.
For a cell $F\in\calS_{\nu}$, let $\link_{\calS_{\nu}}(F)$ be the set of all cells in $\calS_{\nu}$ containing $F$, and we call it the \textit{link of $F$}.
Moreover, we denote by $\sigma(F)$ the smallest face of $P$ containing $F$.

\begin{defi}
For a (possibly empty) cell $F\in \calS_{\nu}$,
the \textit{$h$-polynomial $h(\mathrm{lk}_{\mathcal{S}}(F);t)$ of the link $\link_{\calS_{\nu}}(F)$}
is defined by
\[t^{\dim{P}-\dim{F}}h(\link_{\calS_{\nu}}(F);t^{-1})=\sum_{F'\in\link_{\calS_{\nu}}(F)}g([F,F'];t)(t-1)^{\dim{P}-\dim{F'}}.\]
The \textit{local $h$-polynomial $l_{P}(\calS_{\nu},F;t)$ of $F$ in $\calS_{\nu}$} is defined by
\[l_{P}(\calS_{\nu},F;t)=\sum_{\sigma(F)\prec Q\prec P}(-1)^{\dim{P}-\dim{Q}}h(\link_{\calS_{\nu}|_{Q}}(F);t)g([Q,P]^*;t).\]
\end{defi}

Note that if $\calS_{\nu}$ is the trivial subdivision of $P$ (i.e. $\calS_{\nu}$ consists of all the faces of $P$), we have $h(\link_{\calS_{\nu}}(Q);t)=g([Q,P];t)$ and $l_{P}(\calS_{\nu},Q;t)=0$ for a face $Q$ of $P$.
For $\lambda\in\CC$, $m\in\ZZ_{\geq 0}$ and $v\in mP\cap \ZZ^n$,
we set
\renewcommand{\arraystretch}{1.5}
\begin{align*}
w_{\lambda}(v)=\left\{
\begin{array}{ll}
1&\ (\exp(2\pi\sqrt{-1}\cdot m\nu(\frac{v}{m}))=\lambda)\\
0&\ (\mbox{otherwise}).
\end{array}
\right.
\end{align*}
\renewcommand{\arraystretch}{1}
For $m\in\ZZ$, we define an integer $f_{\lambda}(P,\nu;m)\in\ZZ$ by
\[f_{\lambda}(P,\nu;m):=\sum_{v\in mP\cap\ZZ^n}w_{\lambda}(v).\]
Then $f_{\lambda}(P,\nu;m)$ is a polynomial in $m$ whose degree is less than or equal to $\dim{P}$.
\begin{defi}
\begin{enumerate}
\item[(i)] If $P\neq \emptyset$, we define the \textit{$\lambda$-weighted $h^*$-polynomial} $h^*_{\lambda}(P,\nu;u)\in\ZZ[u]$ by
\[h^*_{\lambda}(P,\nu;u)=(1-u)^{\dim{P}+1}\sum_{m\geq 0}f_{\lambda}(P,\nu;m)u^m.\]
We set $h^*_{1}(\emptyset,\nu;u)=1$ and $h^*_{\lambda}(\emptyset,\nu;u)=0$ for $\lambda\neq 1$. 
\item[(ii)] If $P\neq \emptyset$, we define the \textit{$\lambda$-local weighted $h^*$-polynomial} $l_{\lambda}^*(P,\nu;u)\in\ZZ[u]$ by
\[l_{\lambda}^*(P,\nu;u)=\sum_{Q\prec P}(-1)^{\dim{P}-\dim{Q}}h^*_{\lambda}(Q,\nu|_{Q};u)\cdot g([Q,P]^*;u).\]
We set $l^*_{1}(\emptyset,\nu;u)=1$ and $l^*_{\lambda}(\emptyset,\nu;u)=0$ for $\lambda\neq 1$.

\item[(iii)] We define the \textit{$\lambda$-weighted limit mixed $h^*$-polynomial} $h^*_{\lambda}(P,\nu;u,v)\in\ZZ[u,v]$ by
\[h_{\lambda}^*(P,\nu;u,v):=\sum_{F\in \calS_{\nu}}v^{\dim{F}+1}l^*_{\lambda}(F,\nu|_{F};uv^{-1})\cdot h(\link_{\calS_{\nu}}(F);uv).\]
\item[(iv)] We define the \textit{$\lambda$-local weighted limit mixed $h^*$-polynomial} 
$l^*_{\lambda}(P,\nu;u,v)\in\ZZ[u,v]$ by
\[l^*_{\lambda}(P,\nu;u,v):=\sum_{F\in\calS_{\nu}}v^{\dim{F}+1}l^*_{\lambda}(P,\nu|_{F};uv^{-1})\cdot l_{P}(\calS_{\nu},F;uv).\]
\item[(v)] We define the \textit{$\lambda$-weighted refined limit mixed $h^*$-polynomial} $h^*_{\lambda}(P,\nu;u,v,w)\in\ZZ[u,v,w]$ by
\[h^*_{\lambda}(P,\nu;u,v,w):=\sum_{Q\prec P}w^{\dim{Q}+1}l^*_{\lambda}(Q,\nu|_{Q};u,v)\cdot g([Q,P];uvw^2).\]
\end{enumerate}
\end{defi}

We introduce some properties of these polynomials.

\begin{prop}[see Theorem~4.21 of Stapledon~\cite{SFM}]\label{hformula}
In the situation as above, the following holds. 
\begin{enumerate}
\item[(i)]
We have
\[l^*_{\lambda}(P,\nu;u,1)=l^*_{\lambda}(P,\nu;u),\]
\[h^*_{\lambda}(P,\nu;u,1)=h^{*}_{\lambda}(P,\nu;u)  \mbox{\quad and}\]
\[h^*_{\lambda}(P,\nu;u,v,1)=h^*_{\lambda}(P,\nu;u,v). \]

\item[(ii)]
We have
\[l^*_{\lambda}(P,\nu;u,v)=l^*_{\ov{\lambda}}(P,\nu;v,u)= (uv)^{\dim P+1}l^*_{\lambda}(P,\nu;v^{-1},u^{-1}) \mbox{\quad and}\]
\[h^*_{\lambda}(P,\nu;u,v,w)=h^*_{\ov{\lambda}}(P,\nu;v,u,w)=h^*_{\lambda}(P,\nu;v^{-1},u^{-1},uvw).\]

\item[(iii)]
We have
\[h^*_{\lambda}(P,\nu;u,v)=\sum_{\substack{F\in \calS_{\nu}\\ \sigma(F)=P}}h^*_{\lambda}(F,\nu|_{F};u,v)(uv-1)^{\dim{P}-\dim{F}}.\]
\end{enumerate}
\end{prop}

For the practical computation of the $h^*$-polynomials, the following example is useful.

\begin{ex}[Example~4.23 of Stapledon~\cite{SFM}]
Assume that $P$ is an $n$-dimensional simplex in $\RR^n$ and $\nu$ is a $\QQ$-affine function on it.
For a face $(\emptyset\neq)Q\prec P$, we denote by $C_{Q}\subset \RR^n\times \RR$ the rational polyhedral cone in $\RR^n\times \RR$ generated by the vectors $\{(v,1)\}_{v\in Q}$.
For the empty face $Q=\emptyset$, we set $C_{Q}=\{0\}\subset \RR^n\times \RR$.
Let $v_{0},\dots, v_{n}\in P\times \{1\}$ be the ray generators of edges of $C_{P}$.
We define a finite subset $\mathrm{Box}$ in $\ZZ^n\times \ZZ$ by
\[\mathrm{Box}:=\{v\in C_{P}\cap (\ZZ^n \times \ZZ)\ |\ v=\sum_{i=0}^{n}a_{i}v_{i}, 0\leq a_{i} <1\}.\]
We denote by $\mathrm{pr}\colon \RR^n\times \RR\lr \RR$ the projection.
Then we have
\[h^*(P,\nu;u,v,w)=\sum_{Q\prec P}\sum_{v\in \mathrm{Box}\cap \Relint{(C_{Q})}}
w_{\lambda}(v)u^{\mathrm{pr}(v)}v^{\dim{Q}+1-\mathrm{pr}(v)}w^{\dim{Q}+1} \quad \mbox{and}\]
\[l^*(P,\nu;u,v)=\sum_{v\in \mathrm{Box}\cap \Relint{(C_{P})}}w_{\lambda}(v)u^{\mathrm{pr}(v)}v^{\dim{P}+1-\mathrm{pr}(v)}.\]
\end{ex}

\subsection{The Hodge realizations of the motivic Milnor fibers}\label{chaphreal}
Next, we discuss the Hodge realization of $\calS_{f,0}$.
We recall the proposition stated below for the Hodge realizations of non-degenerate hypersurfaces in the algebraic torus $\CS^n$.
Let $g(x)=\sum_{\beta\in\ZZ^n}c_{\beta}x^\beta\in \nLPoly$ be a non-degenerate Laurent polynomial (see Definition~\ref{nondeg}) with $\dim\mathrm{NP}(g)=n$,
and $\nu$ a $\QQ$-affine function on $P:=\mathrm{NP}(g)$ such that
$\nu(\beta)\in\ZZ$ if $c_{\beta}\neq 0$.
The multiplication by the element in $\CS^n$ which correspond to the group homomorphism
\[\exp(2\pi\sqrt{-1}\nu(\ \cdot\ ))\in\mathrm{Hom}_{\mathrm{group}}(\ZZ^n,\CC^*)\]
defines a $\muhat$-action on the hypersurface  
\[Z^{\circ}:=\{x\in\CS^n\ |\ g(x)=0\}.\]
We thus obtain an element $[Z^{\circ}\circlearrowleft \muhat]$ in $\calM^{\muhat}_{\CC}$
Set
\[\epsilon(\lambda)=\left\{\begin{array}{ll}1 & (\lambda =1)\\\\0 &(\lambda\neq 1).\end{array}\right.\]
\begin{prop}[Stapledon~\cite{SFM}, Matsui-Takeuchi~\cite{MT}]\label{hodgereal}
In the situation as above,
we have
\[uvE_{\lambda}([Z^{\circ}\circlearrowleft\muhat];u,v)=
\epsilon(\lambda)(uv-1)^{n}+(-1)^{n+1}h^{*}_{\lambda}(P,\nu;u,v),\]
for $\lambda\in\CC$.
\end{prop}
Let us remark that in \cite{SFM} and \cite{MT}
for an algebraic variety $X$ with a $\mu_{m}$-action for some $m\in\ZZ_{\geq 1}$ the authors endowed $H^{j}_{c}(X;\QQ)$
with the inverse automorphism of the one which is defined by $(\ref{actionl})$ in Section~\ref{subsecmot}.
Therefore, Proposition~\ref{hodgereal} is slightly different from the original one in \cite{SFM} and \cite{MT}.

Let $f(x)\in\nPoly$ be a non-constant polynomial such that $f(0)=0$ and assume that $f$ is non-degenerate at $0$.
We denote by $P$ the convex hull $\Conv{(\Ntbd\cup\{0\})}$ of $\Ntbd\cup\{0\}$ in $\RR^n$ and
define a piecewise $\QQ$-affine function $\nu$ on $P$ which takes the value $0$ (resp. $1$)
at the origin $0\in\RR^n$ (resp. on $\Conv{(\Ntbd)}$)
such that for any compact face $F$ of $\Ntph$ the restriction $\nu_{F}$ of $\nu$ to $\Delta_{F}$ is linear.
Moreover, for a compact face $F$ of $\Ntph$ let $0_{F}$ be the zero function on $F$.
Then by Theorem~\ref{motivicdesc} and Proposition~\ref{hodgereal}, we can calculate the Hodge realization of the motivic Milnor fiber $S_{f,0}$, and we can describe $E_{\lambda}(F_{f,0};u,v)$ as follows.

\begin{cor}\label{epolycomp1} 
In the situation as above,
for $\lambda\in\CC$ we have
\begin{align*}
uvE_{\lambda}(F_{f,0};u,v)=\sum_{F\prec\Ntph\colon \mbox{\tiny compact}}(-1)^{\dim F}
\left\{
(1-uv)^{s_{F}-\dim{F}}h^*_{\lambda}(F,0_{F};u,v)+\right.\\
\left.(1-uv)^{s_{F}-\dim{F}-1}h^*_{\lambda}(\Delta_{F},\nu_{F};u,v)\right\},
\end{align*}
where in the sum $\Sigma$ the face $F(\neq \emptyset)$ ranges through the compact ones of $\Ntph$.
\end{cor}
Note that if $\lambda\neq 1$, the polynomial $h^*_{\lambda}(F,0_{F};u,v)$ is zero. 
The coefficient of $u^{p}v^{q}$ in $E_{\lambda}(F_{f,0};u,v)$ being an alternating sum of $h^{p,q}_{\lambda}(H^{j}(F_{f,0};\CC))$, for each $j\in\ZZ$ we can not always compute $h^{p,q}_{\lambda}(H^{j}(F_{f,0};\CC))$ by the formula in Corollary~\ref{epolycomp1}.
Recall that if $0\in V$ is an isolated singular point, we have $H^{j}(F_{f,0};\CC)=0$ unless $j=0$ or $n-1$,
and $h^{p,q}_{\lambda}(H^{0}(F_{f,0};\CC))=0$ unless $\lambda=1$ and $(p,q)= (0,0)$.
Therefore, in this case, we can compute each $h^{p,q}_{\lambda}(H^{n-1}(F_{f,0};\CC))$ by our formula.
Even if $f$ is not convenient (in this case, $0\in\zset$ may be a non-isolated singular point),
we will show later that for ``good'' eigenvalues we have $H^{j}(F_{f,0};\CC)_{\lambda}=0\ (j\neq n-1)$
and we can compute $h^{p,q}_{\lambda}(H^{n-1}(F_{f,0};\CC))$ (see Theorem~\ref{main1}).

Let us explain a symmetry of $E_{\lambda}(F_{f,0};u,v)$.
\begin{defi}\label{admissible}
We say that a compact face $F(\neq \emptyset)$ of $\Ntph$ is \textit{extremal} if there exists a non-compact face $G$ of $\Ntph$ such that $F\prec G$ and $G$ is not contained in the boundary $\partial\RRpn$ of $\RRpn$.
If a compact face $F\prec \Ntph$ is not extremal, we say that $F$ is \textit{admissible}.
\end{defi}

\begin{defi}\label{Rf} 
For a non-constant polynomial $f(x)\in\nPoly$ such that $f(0)=0$,
we define a finite subset $R_{f}$ of $\CC$ by
\[R_{f}:=\bigcup_{F\prec \Ntph \sumcom{extremal}}\{\lambda\in\CC\ |\ \lambda^{d_{F}}=1\},\]
where $F$ ranges through the extremal compact faces of $\Ntph$ and
$d_{F}\in\ZZp$ is the lattice distance of $F$ from the origin $0\in\RR^n$.
\end{defi}

\begin{ex}\label{exrf}
Consider the case where $n=2$.
Let $f$ be the polynomial $f(x_{1},x_{2})=x_{1}^{7}+x_{1}^{3}x_{2}+x_{1}^{2}x_{2}^{4}$.
Then the Newton polyhedron $\Ntph\subset \RR^2$ of $f$ is as in Fig.~\ref{fig2}.
The union of the bold lines in $\Ntph$ in it is the Newton boundary $\Ntbd$ of $f$.
In this case, the only extremal face of $\Ntph$ is the $0$-dimensional face $\{(2,4)\}$.
Its lattice distance from the origin $(0,0)$
is $2$.
Hence, we have 
\[R_{f}=\{1,-1\}.\]

\begin{figure}[htbp]
\begin{tikzpicture}[scale = 0.8] 
\draw[help lines] (-1,-1) grid (9,6);
\draw[fill,black, opacity=.5] (9,0) --(7,0)--(3,1)--(2,4)--(2,6)--(9,6)--cycle;
\draw[very thick] (2,4)--(3,1)--(7,0);
\draw[->] (0,-1)--(0,6);
\draw[->] (-1,0)--(9,0);
\draw (0,0) node[below left]{(0,0)};
\draw (7,0) node[below]{(7,0)};
\draw (3,1) node[below left]{(3,1)};
\draw (2,4) node[left]{(2,4)};
\draw (6,4) node{\large $\Gamma_{+}(f)$};
\draw (2,2) node{\large $\Gamma_{f}$};
\filldraw (7,0) circle (3pt);
\filldraw (3,1) circle (3pt);
\filldraw (2,4) circle (3pt);
\end{tikzpicture}
\caption{The Newton polyhedron of $f=x_{1}^{7}+x_{1}^{3}x_{2}+x_{1}^{2}x_{2}^{4}$}\label{fig2}
\end{figure}
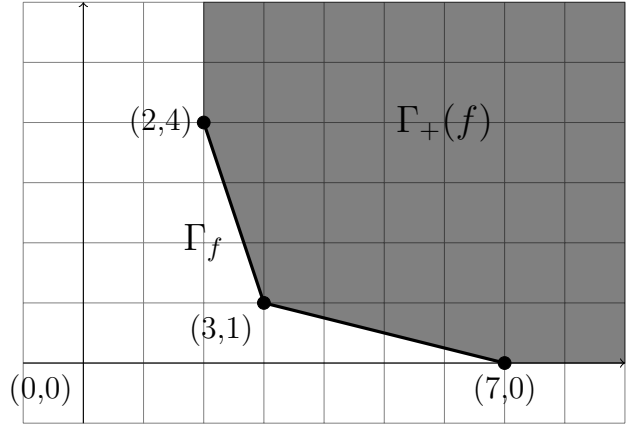

\end{ex}

If $f$ is convenient and $\lambda\neq 1$,
then the weight filtration on $H^{n-1}(F_{f,0})$ is the monodromy filtration and we have
\[h^{p,q}_{\lambda}(H^{n-1}(F_{f,0};\CC))=h^{n-1-q,n-1-p}_{\lambda}(H^{n-1}(F_{f,0};\CC)),\]
for any $p,q\in\ZZ$.
Hence in this case, we have 
\[E_{\lambda}(F_{f,0};u,v)=(uv)^{n-1}E_{\lambda}(F_{f,0};v^{-1},u^{-1}).\]
However, in the case where $f$ is not convenient, this symmetry
does not hold in general.
Nevertheless, by Corollary~\ref{epolycomp1} and some properties in Proposition~\ref{hformula} of the $h^*$-polynomials we can easily see the following results.

\begin{prop}\label{propsymEpoly}
Assume that $f(x)\in\nPoly$ is non-degenerate at $0$ and is not convenient
and $\dim{P}=n$.
Then for any $\lambda\notin R_{f}$ we have
\[
uvE_{\lambda}(F_{f,0};u,v)
=(-1)^{n-1}l^{*}_{\lambda}(P,\nu;u,v).
\]
In particular, for $\lambda\notin R_{f}$ we have the symmetry
\[E_{\lambda}(F_{f,0};u,v)=(uv)^{n-1}E_{\lambda}(F_{f,0};v^{-1},u^{-1}).\]
\end{prop}

\section{Main theorem}\label{chapmain}
Let $f(x)\in\nPoly$ be a non-constant polynomial such that $f(0)=0$ and set $V:=f^{-1}(0)\subset \CC^n$ as before.
Throughout this section, we assume that $f$ is non-degenerate at $0$ (see Definition~\ref{nondegatzero}).
Since our assertion below will become trivial if $\dim{P}(=\dim{\Conv(\Ntbd\cup\{0\})})<n$,
in what follows we assume also that $\dim{P}=n$.
Note that we do not assume $f$ is convenient here, and therefore the origin $0\in\CC^n$ may be a non-isolated singular point of $V$ in general.
So we can not expect to have the concentration
$\tl{H}^{j}(F_{f,0};\CC)\simeq 0\ (j\neq n-1)$.
However, we will prove the following result.

\begin{thm}\label{main1}
In the situation as above, for any $\lambda\notin R_{f}$ (see Definition~\ref{Rf}) we have a concentration
\[\tl{H}^{j}(F_{f,0};\CC)_{\lambda}=0\qquad(j\neq n-1).\]
\end{thm}

By this theorem and Theorem~\ref{Varform},
we obtain the following corollary.
\begin{cor}\label{multico}
In the situation of Theorem~\ref{main1},
for any $\lambda\notin R_{f}$ the multiplicity of the eigenvalue $\lambda$ in the Milnor monodromy $\Phi_{n-1}$ is equal to that of the factor $(1-\lambda t)$ in a rational function
\[\prod_{\emptyset\neq I\subset \{1,\dots,n\}}\prod_{i=1}^{k_{I}}(1-t^{d_{I,i}})^{(-1)^{n-|I|}\mathrm{Vol}_{\ZZ}(\Gamma_{I,i})}.\]
For the definitions of $k_{I}$, $\Gamma_{I,i}$, $d_{I,i}$ and $\mathrm{Vol}_{\ZZ}(\Gamma_{I,i})$, see Section~\ref{chapmil}.   
\end{cor}

For the proof of Theorem~\ref{main1},
we need the following proposition.
\begin{prop}\label{compare1prop}
In this situation of Theorem~\ref{main1}, for any $\lambda\notin R_{f}$, $k\in\ZZ$
and the inclusion map $j_{0}\colon\{0\}\hlr V$
the natural morphism in $\DBC{\{0\}}=\DB{\mathrm{Mod}(\CC)}\colon$ 
\begin{align*}
j_{0}^{!}({}^p\psi_{f,\lambda}(\CC_{\CC^n}[n]))\simar j_{0}^{-1}({}^p\psi_{f,\lambda}(\CC_{\CC^n}[n]))
\end{align*}
is an isomorphism.
\end{prop}

\begin{proof}[Proof of Proposition~\ref{compare1prop}]
It suffices to show that for any $\lambda\notin R_{f}$ the morphism
\begin{align}\label{firsteq}
j_{0}^{!}\psi_{f,\lambda}(\CC_{\CC^n})\lr j_{0}^{-1}\psi_{f,\lambda}(\CC_{\CC^n})
\end{align}
is an isomorphism.
We denote by $\Sigma_{0}$ the normal fan of the Newton polyhedron $\Gamma_{+}(f)$.
For a subset $I\subset \{1,2,\dots,n\}$ we set $\RR^{I}_{\geq 0}:=\RR^I\cap\RR^n_{\geq 0}$.
We construct a smooth subdivision $\Sigma$ of $\Sigma_{0}$ without subdividing the cones in $\Sigma_{0}$ of the type $\RR^{I}_{\geq 0}.$
Let $X_{\Sigma}$ be the toric variety associated with $\Sigma$.
Recall that $X_{\Sigma}$ contains $T:=\CS^n$ as an open dense subset
and it acts naturally on $X_{\Sigma}$ itself.
For a cone $\sigma\in\Sigma$, we denote by $T_{\sigma}$ the $T$-orbit in $X_{\Sigma}$ associated with it.
Let $\Sigma_{1}$ be the fan formed by all the faces of $\RR^n_{\geq 0}$.
Then the toric variety associated with it is $\CC^n$.
Moreover, the morphism of fans $\Sigma\lr \Sigma_{1}$ induces a proper morphism
\[\pi\colon X_{\Sigma}\lr \CC^n\]
of toric varieties.
Set $V':=\pi^{-1}(V)$.

\begin{lem}\label{pushconst1}
In the situation as above, we have an isomorphism
\[\psi_{f,\lambda}(\CC_{\CC^n})\simeq \psi_{f,\lambda}(\DR\pi_{*}\CC_{X_{\Sigma}})\] in $\DBC{V}$ for any $\lambda\in\CC$.
\end{lem}
\begin{proof}
By the construction of $\Sigma$, one can easily show that
the morphism $\pi$ induces an isomorphism
\[X_{\Sigma}\setminus V'\simar \CC^n\setminus V.\]
Hence we obtain an isomorphism
\[(\DR\pi_{*}\CC_{X_{\Sigma}})|_{\CC^n\setminus V}\simeq \CC_{\CC^n\setminus V}\]
in $\DBC{\CC^n\setminus V}$.
Now the desired assertion immediately follows from the definition of the nearby cycle functor.
\end{proof}

In what follows, we fix $\lambda\notin R_{f}$.
Consider the following distinguished triangle in $\DBC{\{0\}}\simeq \mathrm{D^b}(\mathrm{Mod}(\CC))$:
\begin{align*}
(\DR\Gamma_{\{0\}}\psi_{f,\lambda}(\CC_{\CC^n}))_{0}\lr
\psi_{f,\lambda}(\CC_{\CC^n})_{0}\lr
(\DR\Gamma_{V\setminus\{0\}}\psi_{f,\lambda}(\CC_{\CC^n}))_{0}\overset{+1}{\lr}.
\end{align*}
The first arrow in it coincides with the natural morphism 
\begin{align}\label{ar1}
j^{!}_{0}\psi_{f,\lambda}(\CC_{\CC^n})\lr j^{-1}_{0}\psi_{f,\lambda}(\CC_{\CC^n}).
\end{align}
Therefore, it is enough to show that $(\DR\Gamma_{V\setminus\{0\}}\psi_{f,\lambda}(\CC_{\CC^n}))_{0}$ is isomorphic to $0$.
Consider the Cartesian diagram:
\[
\xymatrix{
V\setminus\{0\} \ar@{}[dr]|-{\square}\ar@{^(->}[r]^-{i} &V\\
V'\setminus\pi^{-1}(0)\ar[u]^{\pi''}\ar@{^(->}[r]^-{i'}& V'\ar[u]_{\pi'},
}
\]
where $i, i'$ are the inclusions and $\pi', \pi''$ are the restrictions of $\pi$.
Since $\pi'$ and $\pi''$ are proper, we obtain the following isomorphisms,
where in the third isomorphism we used Proposition~4.2.11 of Dimca~\cite{DIMCA}: 
\begin{align}
(\DR\Gamma_{V\setminus\{0\}}\psi_{f,\lambda}(\CC_{\CC^n}))_{0}&=
(\DR i_{*}i^{-1}\psi_{f,\lambda}(\CC_{\CC^n}))_{0}\notag\\\notag
& \simeq (\DR i_{*}i^{-1}\psi_{f,\lambda}(\DR\pi_{*}\CC_{X_{\Sigma}}))_{0} \qquad (\mbox{by Lemma~\ref{pushconst1}})\\\notag
& \simeq (\DR i_{*}i^{-1}\DR\pi'_{*}\psi_{f\circ\pi,\lambda}(\CC_{X_{\Sigma}}))_{0}\\\notag
& \simeq (\DR i_{*}\DR\pi''_{*}i'^{-1}\psi_{f\circ\pi,\lambda}(\CC_{X_{\Sigma}}))_{0}\\\notag
&\simeq (\DR \pi'_{*}\DR i'_{*}i'^{-1}\psi_{f\circ\pi,\lambda}(\CC_{X_{\Sigma}}))_{0}\\
&\simeq \DR\Gamma(\pi^{-1}(0);(\DR i'_{*}i'^{-1}\psi_{f\circ\pi,\lambda}(\CC_{X_{\Sigma}}))|_{\pi^{-1}(0)})\\
&\simeq \DR\Gamma(\pi^{-1}(0);(\DR \Gamma_{V'\setminus \pi^{-1}(0)}(\psi_{f\circ\pi,\lambda}(\CC_{X_{\Sigma}})))|_{\pi^{-1}(0)})\label{vanish1}
\end{align}
in $\mathrm{D^{b}}(\mathrm{Mod}(\CC))$.
The following argument is inspired by the proof of Theorem~3.17 in Saito-Takeuchi~\cite{SaiTake}.
Let $\rho_{1},\dots, \rho_{N}$ be the rays in $\Sigma$.
By abuse of notation, we will use the same symbol $\rho_{i}$ for the primitive vector on the ray $\rho_{i}$. 
For $1\leq i\leq N$, set
\[m_{\rho_{i}}:=\min_{\alpha \in \Ntph}\inpro{\alpha}{\rho_{i}}\in\ZZ_{\geq 0},\]
where $\inpro{\alpha}{\rho_{i}}$ is the inner product of $\alpha$ and $\rho_{i}$ in $\RR^n$.
We may assume that for some $1\leq l\leq N$ we have
$m_{\rho_{i}}\neq 0$ and $\lambda^{m_{\rho_{i}}}=1$ if and only if $1\leq i\leq l$,
and for some $l\leq l' \leq N$ we have
$m_{\rho_{i}}=0$ if and only if $l'+1\leq i\leq N$.
For a cone $\sigma\in \Sigma$ containing some $\rho_{i}$ with $1\leq i\leq l'$, we also set
\[m_{\sigma}:=\gcd_{\substack{1\leq i\leq l'\\\rho_{i}\prec \sigma}}(m_{\rho_{i}})\in \ZZ_{\geq 1}.\]
For $1\leq i\leq l'$ we denote by $E_{i}$ the closure $\overline{T_{\rho_{i}}}$ of $T_{\rho_{i}}$ in $X_{\Sigma}$.
Note that the order of zeros of $f\circ \pi$ along the divisor $E_{i}$ is equal to $m_{\rho_{i}}$.
Let $Z$ be the strict transform of $V$ in $X_{\Sigma}$.
Then $V'=(f\circ \pi)^{-1}(0)$ has the following form
\[V'=E_{1}\cup\dots\cup E_{l'}\cup Z.\]
Since $f$ is non-degenerate at $0$, the divisor $V'$ is normal crossing in a neighborhood of $\pi^{-1}(0)$.
In particular, $Z$ is smooth there.
Thus, for our discussion below we may assume $V'$ is normal crossing.
For a subset $I\subset \{1,\dots,l\}$,
we set
\[E_{I}:=\bigcap_{i\in I} E_{i}, \quad E_{I}^{\circ}:=E_{I}\setminus(\bigcup_{\substack{1\leq i\leq l'\\ i\notin I}}E_{i}\cup Z) \mbox{\quad and\quad} U_{I}:=E_{I}\setminus(\bigcup_{l+1\leq i\leq l'}E_{i}\cup Z).\]
Moreover, we write $i_{I}$ and $j_{I}$ for the inclusion maps $i_{I}\colon E_{I}^{\circ}\hlr E_{I}$ and $j_{I}\colon U_{I}\hlr E_{I}$ respectively.
We denote by $\iota$ the inclusion map $\iota\colon X_{\Sigma}\setminus V'\hlr X_{\Sigma}$ and define a sheaf $\calF_{\lambda}$ on $X_{\Sigma}$ by
\[\calF_{\lambda}:=\iota_{*}(f\circ\pi|_{X_{\Sigma}\setminus V'})^{-1}\calL_{\lambda^{-1}},\]
where $\calL_{\lambda^{-1}}$ is the $\CC$-local system on $\CC^*$ of rank $1$ whose monodromy is given by the multiplication by $\lambda^{-1}$. 
Since for $1\leq i\leq l'$ the monodromy of the local system $(f\circ \pi)^{-1}\calL_{\lambda^{-1}}$ around the divisor $E_{i}$ is given by $\lambda^{-m_{\rho_{i}}}$ , the restriction of $\calF_{\lambda}$ to $U_{I}$ is a local system of rank $1$.
For $I\subset \{1,\dots,n\}$, we set $\calF_{\lambda, I}=\calF_{\lambda}|_{E_{I}}$ and
write $i_{E_{I}}$ for the inclusion $i_{E_{I}}\colon E_{I}\hlr V'$.
Then, by the primitive decomposition of ${}^p\psi_{f\circ\pi,\lambda}(\CC_{X_{\Sigma}}[n])$ each graded piece of ${}^p\psi_{f\circ\pi,\lambda}(\CC_{X_{\Sigma}}[n])$ with respect to the filtration $W_{\bullet}{}^p\psi_{f\circ\pi,\lambda}(\CC_{X_{\Sigma}}[n])$ is a direct sum of some perverse sheaves 
\[{i_{E_{I}}}_{*}{j_{I}}_{!}j_{I}^{-1}\calF_{\lambda,I}[n-|I|](\simeq {i_{E_{I}}}_{*}{j_{I}}_{*}j_{I}^{-1}\calF_{\lambda,I}[n-|I|])\] for $I\subset \{1,\dots,l\}$ (see Section~1.4 of \cite{somecons} and Section~4.2 of \cite{IGUSA}). 
Therefore, to show that $(\ref{vanish1})$ is isomorphic to $0$ in $\mathrm{D^b}(\mathrm{Mod}(\CC))$,
it suffices to show that for each $I\subset \{1,\dots,l\}$, we have
\begin{align}\label{vanish2}
\DR\Gamma(\pi^{-1}(0);(\DR \Gamma_{V'\setminus \pi^{-1}(0)}({i_{E_{I}}}_{*}{j_{I}}_{!}j_{I}^{-1}\calF_{\lambda,I}))|_{\pi^{-1}(0)})\simeq 0.
\end{align}
Fix $I\subset \{1,\dots, l\}$ such that $E_{I}\neq \emptyset$.
If for some $i\in I$ we have $E_{i}\subset \pi^{-1}(0)$,
the sheaf ${i'}^{-1}{i_{E_{I}}}_{*}{j_{I}}_{!}j_{I}^{-1}\calF_{\lambda,I}$ is zero.
For this reason, in what follows we may assume that $E_{i}\not\subset \pi^{-1}(0)$ for any $i\in I$.
Namely, the ray $\rho_{i}$ is contained in the boundary $\partial\RRpn$ of $\RRpn$ for any $i\in I$.
By the property of $\calF_{\lambda}$ stated above,
we can easily see the isomorphisms on $E_{I}$:
\begin{align}
{j_{I}}_{!}j_{I}^{-1}\calF_{\lambda, I}\simeq {j_{I}}_{*}j_{I}^{-1}\calF_{\lambda,I}\simeq
\DR{j_{I}}_{*}{j_{I}}^{-1}\calF_{\lambda,I}\simeq \DR{\Gamma}_{U_{I}}(\calF_{\lambda,I}).
\end{align}
We thus obtain
\[\DR \Gamma_{V'\setminus \pi^{-1}(0)}({i_{E_{I}}}_{*}{j_{I}}_{!}j_{I}^{-1}\calF_{\lambda,I})\simeq {i_{E_{I}}}_{*}\DR\Gamma_{U_{I}\setminus \pi^{-1}(0)}(\calF_{\lambda,I})\simeq 
{i_{E_{I}}}_{*}{j_{I}'}_{*}{j'_{I}}^{-1}\calF_{\lambda,I},\]
where we write $j_{I}'$ for the inclusion $j_{I}'\colon U_{I}\setminus \pi^{-1}(0)\hlr E_{I}$.
Note that its restriction to the outside of $U_{I}$ in a small neighborhood of $\pi^{-1}(0)$ is zero.
Therefore, to show (\ref{vanish2}), it is enough to show that
\begin{align}\label{vanish3}
\DR\Gamma_{c}(U_{I}\cap \pi^{-1}(0); (\DR {j'_I}_{*}{j'_{I}}^{-1}\calF_{\lambda,I})|_{U_{I}\cap \pi^{-1}(0)})\simeq 0.
\end{align}
Since $E_{I}\neq \emptyset$, the cone $\tau$ generated by the rays $\{\rho_{i}\}_{i\in I}$ is in $\Sigma$.
The subvariety $U_{I}\cap \pi^{-1}(0)$ of $X_{\Sigma}$ is the union of $T_{\sigma}\setminus Z$ for the cones $\sigma\in\Sigma$ satisfying the following condition $(\star)$:
\[
(\star)\left\{\begin{array}{cl}
(\mathrm{i})&\mbox{$\tau$ is a face of $\sigma$},\\
(\mathrm{ii})&\mbox{$\Relint{\sigma}\subset \Int{\RRpn}$ \quad and}\\
(\mathrm{iii})&\mbox{any ray $\rho_{i}\in\Sigma$ contained in $\sigma$ satisfies $\lambda^{m_{\rho_{i}}}=1$},
\end{array}
\right.
\]
where $\Relint{\sigma}$ is the relative interior of $\sigma$ and $\Int{\RRpn}$ is the interior of $\RRpn$.
Note that such $\sigma$ may contain some rays $\rho_{i}$ such that $i>l'$.
Therefore, to show (\ref{vanish3}), we shall show that
for any $\sigma\in\Sigma$ with the condition ($\star$) we have 
\begin{align}\label{vanish4}
\DR\Gamma_{c}(T_{\sigma}\setminus Z; (\DR{j'_{I}}_{*}{j'_{I}}^{-1}\calF_{\lambda,I})|_{T_{\sigma}\setminus Z})\simeq 0.
\end{align}
In what follows, we fix a cone $\sigma\in\Sigma$ with the condition ($\star$).
Let $\tl{\sigma}\in \Sigma_{0}$ be the unique cone in $\Sigma_{0}$ such that
$\Relint{\sigma}\subset \Int{\tl{\sigma}}$ 
and $F(\tl{\sigma})\prec\Ntph$ the face of $\Ntph$ which corresponds to it. 
By the condition $\Relint{\sigma}\subset \Int{\RR^n_{\geq 0}}$, the face $F(\tl{\sigma})$ is compact.
We denote by $d_{\Ftsigma}$ the lattice distance of $\Ftsigma$ from the origin $0\in\RR^n$.
Then we can easily show the following assertion.
\begin{lem}\label{hullsingmaeq}
Assume that $\dim\sigma=\dim \tl{\sigma}$.
Then we have $m_{\sigma}=d_{\Ftsigma}$.
\end{lem}
Suppose that $\dim\sigma=\dim \tl{\sigma}$.
Since for any $i\in I$ we have $\rho_{i}\subset \partial \RR^n_{\geq 0}$ and $m_{\rho_{i}}>0$,
there exists a non-compact face $G$ of $\Ntph$ containing $\Ftsigma$ and $G\not\subset\partial\RRpn$.
Then, by Lemma~\ref{hullsingmaeq} and the assumption that $\lambda\notin R_{f}$,
we have $\lambda^{m_{\sigma}}=\lambda^{d_{F(\tl{\sigma})}}\neq 1$.
Therefore, there exists a ray $\rho_{i}$ in $\sigma$ such that $\lambda^{m_{\rho_{i}}}\neq 1$.
This contradicts our condition $(\star)$.
Hence we have 
\[\dim\sigma<\dim \tl{\sigma}.\]
Take a cone $\sigma'\in \Sigma$ such that
$\sigma\prec \sigma' \subset \tl{\sigma}$ and $\dim\sigma'=\dim \tl{\sigma}$.
Then by the argument as above,
we have $\lambda^{m_{\sigma'}}\neq 1$.
Thus, by the condition $(\star)$ it follows that there exists a ray $\rho_{i}\prec \sigma'$ such that $\rho_{i}\not\prec \sigma$ and
\begin{align}\label{neq1nosiki}
\lambda^{m_{\rho_{i}}}\neq 1.
\end{align}
Moreover, take a $n$-dimensional cone $\sigma''\in\Sigma$ such that $\sigma'\prec\sigma''$.
Let $\Edge(\tau)$, $\Edge(\sigma)$, $\Edge(\sigma')$ and $\Edge(\sigma'')$
be the sets of edges (i.e rays $\rho_{i}$) of the smooth cones $\tau$, $\sigma$, $\sigma'$ and $\sigma''$ respectively.
We assume that for some $1\leq i_{1}\leq i_{2}<i_{3}\leq i_{4}\leq n$ we have
\begin{align*}
\Edge(\tau)&=\{\xi_{1},\dots,\xi_{i_{1}}\},\\
\Edge(\sigma)&=\{\xi_{1},\dots, \xi_{i_{2}}\},\\
\Edge(\sigma')&=\{\xi_{1},\dots,\xi_{i_{3}}\},\\
\Edge(\sigma'')&=\{\xi_{1},\dots,\xi_{n}\}
\end{align*}
($\xi_{i}\in\Sigma$).
Set $s_{1}:=i_{1}$, $s_{2}:=i_{2}-i_{1}$, $s_{3}:=i_{3}-i_{2}$, $s_{4}:=n-i_{3}$.
Note that by the condition $\sigma\Precneq \sigma'$ we have $s_{3}>0$.
Since $\sigma''$ is a smooth cone,
the affine open subset $\CC^n(\sigma'')\simeq\CC^n$ of $X_{\Sigma}$ associated with $\sigma''$ has a natural decomposition:
\[\CC^n(\sigma'')=\CC^{s_{1}}\times \CC^{s_{2}}\times \CC^{s_{3}}\times \CC^{s_{4}}.\]
Let
\[(x_{1},\dots,x_{s_{1}},y_{1},\dots,y_{s_{2}},z_{1},\dots, z_{s_{3}},w_{1},\dots,w_{s_{4}})\]
be the corresponding coordinates of $\CC^n(\sigma'')$.
In $\CC(\sigma'')\simeq \CC^n$, we have
\begin{align*}
E_{I}&=\{0\}\times \CC^{s_{2}}\times \CC^{s_{3}}\times \CC^{s_{4}} \qquad\mbox{and}\\
T_{\sigma}&=\{0\}\times\{0\}\times\CS^{s_{3}}\times\CS^{s_{4}}.
\end{align*}
For any $i_{2}+1\leq i\leq i_{3}$ the function $\inpro{\xi_{i}}{\cdot}$ is constant on $F(\tl{\sigma})$.
Since $f$ is non-degenerate at $0$,
$T_{\sigma}\cap Z$ is smooth and its defining polynomial in $T_{\sigma}$ can be described by
\[f\circ \pi=z_{1}^{m_{\xi_{i_{2}+1}}}\dots z_{s_{3}}^{m_{\xi_{i_{3}}}}g(w_{1},\dots,w_{s_{4}}),\]
where $g(w_{1},\dots, w_{s_{4}})\in\CC[w_{1},\dots,w_{s_{4}}]$.
We denote by $W$ the zero set of $g(w_{1},\dots,w_{s_{4}})$ in $\CS^{s_{4}}$.
Then we have
\[T_{\sigma}\cap Z=\{0\}\times \{0\}\times \CS^{s_{3}}\times W,\]
and
\[T_{\sigma}\setminus Z=\{0\}\times \{0\}\times \CS^{s_{3}}\times (\CS^{s_{4}}\setminus W).\]
Let $p_{3}$ be the projection $p_{3}\colon T_{\sigma}\setminus Z=\CS^{s_{3}}\times (\CS^{s_{4}}\setminus W)\lr\CS^{s_{3}}$,
and $p_{4}$ be the projection $p_{4}\colon T_{\sigma}\setminus Z=\CS^{s_{3}}\times (\CS^{s_{4}}\setminus W)\lr \CS^{s_{4}}\setminus W$. 
We define $\calL_{3}$ by the $\CC$-local system on $\CS^{s_{3}}$ of rank $1$ whose monodromy around the divisor $\{(z_{1},\dots,z_{s_{3}})\in \CC^{s_{3}}\ |\ z_{t}=0\}$ is given by the multiplication by $\lambda^{-m_{\xi_{(i_{2}+t)}}}$ for each $1\leq t\leq s_{3}$.
Note that by $(\ref{neq1nosiki})$ there exists $i_{2}+1\leq i\leq i_{3}$ such that
\begin{align}\label{neq1nosiki3}
\lambda^{-m_{\xi_{i}}}\neq 1.
\end{align}
By the definition of $\calF_{\lambda,I}$, one can show that there exist a $\CC$-local system $\calL_{4}$ on $\CS^{s_{4}}\setminus W$ and a complex $C^{\bullet}$ of $\CC$-vector spaces such that
\[(\DR{j'_{I}}_{*}{j'_{I}}^{-1}\calF_{I,\lambda})|_{T_{\sigma}\setminus Z}\simeq C^{\bullet}\otimes_{\CC}p_{3}^{-1}\calL_{3}\otimes _{\CC}p_{4}^{-1}\calL_{4}.\]
Recall that for any non-trivial $\CC$-local system $\calL$ on $\CC^*$ of rank $1$, we have $H^{j}(\CC^*;\calL)\simeq 0$ for all $j\in\ZZ$.
Hence, by the K\"unneth formula and (\ref{neq1nosiki3})
we deduce the vanishing $(\ref{vanish4})$. 
This completes the proof that the morphism~(\ref{firsteq}) is isomorphism
and the proof of Proposition~\ref{compare1prop}.
\end{proof}

\begin{proof}[Proof of Theorem~\ref{main1}]
Recall that the complex 
$^{p}\psi_{f,\lambda}(\CC_{\CC^n}{[n]})=\psi_{f,\lambda}(\CC_{\CC^n}{[n-1]}) \in \DBC{V}$ is a perverse sheaf.
By the fact that the functor $j_{0}^{-1} \colon \DBC{V}\lr \DBC{\{0\}}$ (resp. $j_{0}^{!} \colon \DBC{V}\lr \DBC{\{0\}}$) is right (resp. left) $t$-exact,
it follows from Proposition~\ref{compare1prop} that $j_{0}^{-1}({}^{p}\psi_{f,\lambda}(\CC_{\CC^n}[n]))$ is a perverse sheaf on $\{0\}$.
Hence the cohomology group 
\[H^{j}(j_{0}^{-1}({}^{p}\psi_{f,\lambda}(\CC_{\CC^n}[n])))\simeq H^{j+n-1}(F_{f,0};\CC)_{\lambda}\] vanishes for $j\neq 0$.
We thus obtain the desired concentration.
\end{proof}

Recall that the cohomology groups $H^{j}(F_{f,0};\QQ)$ of the Milnor fiber are endowed with mixed Hodge structures.
Since we do not assume here that $f$ is convenient,
we can not expect that their weight filtrations coincide with the monodromy filtrations in general.
However, we will obtain the following result.

\begin{thm}\label{main2}
In the situation of Theorem~\ref{main1}, for any $\lambda\notin R_{f}$
the filtration on $H^{n-1}(F_{f,0};\CC)_{\lambda}$ induced by the weight filtration on $H^{n-1}(F_{f,0};\QQ)$
coincides with the monodromy filtration of the logarithm of the unipotent part of $\Phi_{n-1,\lambda}$ centered at $n-1$.
\end{thm}

\begin{rem}
Theorems~\ref{main1} and \ref{main2}
explain the reason why the coefficients of the polynomial $E_{\lambda}(F_{f,0};u,v)$ for $\lambda\notin R_{f}$ satisfy the symmetry in Proposition~\ref{propsymEpoly}.
\end{rem}

For the proof of Theorems~\ref{main2}, 
we need a generalization of Proposition~\ref{compare1prop} below.
In what follows, we shall freely use the notations in the proof of Proposition~\ref{compare1prop}.
Let $W_{\bullet}({}^p\psi_{f,\lambda}(\CC_{\CC^n}[n]))$ be the filtration on the perverse sheaf ${}^p\psi_{f,\lambda}(\CC_{\CC^n}[n])$ defined by the weight filtration of the mixed Hodge module $\psi_{f}^{H}(\QQ^{H}_{\CC^n}[n])$.
More generally, by using the exact functors $W_{k}\colon \mathrm{MHM}(X)\to \mathrm{MHM}(X)$ for an object $\calM^{\bullet}\in \mathrm{D^{b}MHM}(X)$ we can define new ones $W_{k}\calM^{\bullet}$ in $\mathrm{D^{b}MHM}(X)$.

\begin{prop}\label{compare1propgen} 
In this situation of Theorems~\ref{main2}, for any $\lambda\notin R_{f}$, $k\in\ZZ$
and the inclusion map $j_{0}\colon\{0\}\hlr V$
the natural morphism in $\DBC{\{0\}}=\DB{\mathrm{Mod}(\CC)}\colon$ 
\begin{align*}
j_{0}^{!}W_{k}({}^p\psi_{f,\lambda}(\CC_{\CC^n}[n]))\lr j_{0}^{-1}W_{k}({}^p\psi_{f,\lambda}(\CC_{\CC^n}[n]))
\end{align*}
is an isomorphism.
\end{prop}

\begin{proof}
It is enough to show that
\begin{align} \notag
\DR\Gamma_{V\setminus \{0\}}(W_{k}({}^p\psi_{f,\lambda}(\CC_{\CC^n}[n])))_{0}\simeq 0
\end{align}
for any $k\in\ZZ$.
Moreover, since by Lemma~\ref{pushconst1} we have $\DR\pi'_{*}({}^p\psi_{f\circ\pi,\lambda}(\CC_{X_{\Sigma}}[n]))\simeq 
{}^p\psi_{f,\lambda}(\CC_{\CC^n}[n])$,
it suffices to show that
\begin{align}\notag
\DR\Gamma_{V\setminus \{0\}}(W_{k}\DR\pi'_{*}\GR^{W}_{i}({}^p\psi_{f\circ\pi,\lambda}(\CC_{X_{\Sigma}}[n])))_{0}\simeq 0
\end{align}
for any $i, k\in\ZZ$.
Thus we have only to show that
\begin{align}\label{simesu2}
\DR\Gamma_{V\setminus \{0\}}(W_{k}(\pcalH^{j}\DR\pi'_{*}\GR^{W}_{i}({}^p\psi_{f\circ\pi,\lambda}(\CC_{X_{\Sigma}}[n]))))_{0}\simeq 0
\end{align}
for any $i,j,k\in\ZZ$, where for $\calF^{\bullet}\in \mathrm{D^{b}_{c}(V)}$ we denote by ${}^p \calH^{j}(\calF^{\bullet})$ the $j$-th perverse cohomology of $\calF^{\bullet}$.
Note that $\pcalH^{j}\DR\pi'_{*}\GR^{W}_{i}({}^p\psi_{f\circ\pi,\lambda}(\CC_{X_{\Sigma}}[n]))$ has a pure weight $i+j$.
Therefore, we have
\begin{align*}
W_{k}(\pcalH^{j}\DR\pi'_{*}\GR^{W}_{i}({}^p\psi_{f\circ\pi,\lambda}(\CC_{X_{\Sigma}}[n])))
\simeq \left\{
\begin{array}{ll}
0&(k<i+j)\\\\
\pcalH^{j}\DR\pi'_{*}\GR^{W}_{i}({}^p\psi_{f\circ\pi,\lambda}(\CC_{X_{\Sigma}}[n]))&(k\geq i+j).
\end{array}
\right.
\end{align*}
Eventually, to show the vanishing (\ref{simesu2}), it is enough to show that for any $i,j\in\ZZ$ we have
\begin{align}\label{simesu3}
\DR\Gamma_{V\setminus \{0\}}(\pcalH^{j}\DR\pi'_{*}\GR^{W}_{i}({}^p\psi_{f\circ\pi,\lambda}(\CC_{X_{\Sigma}}[n])))_{0}\simeq 0.
\end{align}
Note that a perverse sheaf $\GR^{W}_{i}({}^p\psi_{f\circ\pi,\lambda}(\CC_{X_{\Sigma}}[n])\oplus \GR^{W}_{i}({}^p\psi_{f\circ\pi,\bar{\lambda}}(\CC_{X_{\Sigma}}[n])$ is the complexification of the underlying perverse sheaf of a pure Hodge module.
Then, by the decomposition theorem for the proper map $\pi'$ and this perverse sheaf,
$\DR\pi'_{*}\GR^{W}_{i}({}^p\psi_{f\circ\pi,\lambda}(\CC_{X_{\Sigma}}[n])$ is decomposed into a direct sum of perverse sheaves with some shifts.
Therefore, to show (\ref{simesu3}), it remains for us to show that
\begin{align}\notag
\DR\Gamma_{V\setminus \{0\}}(\DR\pi'_{*}\GR^{W}_{i}({}^p\psi_{f\circ\pi,\lambda}(\CC_{X_{\Sigma}}[n])))_{0}\simeq 0
\end{align}
for any $i\in\ZZ$.
This was already proved in the proof of Proposition~\ref{compare1prop}.
\end{proof}

For the proof of Theorem~\ref{main2}, we also need the following lemma.
\begin{lem}\label{compare3} 
For any $\lambda\notin R_{f}$ and $k\in\ZZ$ we have an isomorphism in $\DBC{\{0\}}\simeq \mathrm{D^b}(\mathrm{Mod}(\CC))\colon$ 
\begin{align*}
W_{k}j_{0}^{-1}({}^p\psi_{f,\lambda}(\CC_{\CC^n}[n]))\simeq j_{0}^{-1}W_{k}\left({}^p\psi_{f,\lambda}(\CC_{\CC^n}[n])\right).
\end{align*}
\end{lem}
\begin{proof}
For $k\in\ZZ$ we have an exact sequence in $\Perv{V}$:
\begin{align}\label{con3ex}
0\to W_{k}({}^p\psi_{f,\lambda}(\CC_{\CC^n}[n]))\to {}^p\psi_{f,\lambda}(\CC_{\CC^n}[n])\to {}^p\psi_{f,\lambda}(\CC_{\CC^n}[n])/W_{k}({}^p\psi_{f,\lambda}(\CC_{\CC^n}[n]))\to 0.
\end{align}
By Proposition~\ref{compare1prop} and this sequence, 
the natural morphism
\begin{align*}
j^{!}_{0}\left({}^p\psi_{f,\lambda}(\CC_{\CC^n}[n])/W_{k}({}^p\psi_{f,\lambda}(\CC_{\CC^n}[n]))\right)\lr j^{-1}_{0}({}^p\psi_{f,\lambda}(\CC_{\CC^n}[n])/W_{k}({}^p\psi_{f,\lambda}(\CC_{\CC^n}[n])))
\end{align*}
is an isomorphism.
Thus, by the proof of Theorem~\ref{main1} the complexes 
$j^{-1}_{0}({}^p\psi_{f,\lambda}(\CC_{\CC^n}[n]))$,\ 
$j^{-1}_{0}W_{k}({}^p\psi_{f,\lambda}(\CC_{\CC^n}[n]))$ and
$j^{-1}_{0}({}^p\psi_{f,\lambda}(\CC_{\CC^n}[n])/W_{k}({}^p\psi_{f,\lambda}(\CC_{\CC^n}[n])))$ are perverse sheaves on $\{0\}$, i.e. their (perverse) cohomologies are concentrated in the degree $0$.
Therefore, applying the functor $j_{0}^{-1}$ to the sequence (\ref{con3ex}),
we obtain an exact sequence
\begin{align}\label{con3ex2}
0\to j_{0}^{-1}W_{k}({}^p\psi_{f,\lambda}(\CC_{\CC^n}[n]))\to j_{0}^{-1}({}^p\psi_{f,\lambda}(\CC_{\CC^n}[n]))\to j_{0}^{-1}({}^p\psi_{f,\lambda}(\CC_{\CC^n}[n])/W_{k}({}^p\psi_{f,\lambda}(\CC_{\CC^n}[n])))\to 0
\end{align}
in $\Perv{\{0\}}\simeq \Mod{\CC}$.
Since the functor $j_{0}^{!}$ preserves the property that a complex of mixed Hodge modules has weights $> k$,
$j_{0}^{-1}({}^p\psi_{f,\lambda}(\CC_{\CC^n}[n])/W_{k}({}^p\psi_{f,\lambda}(\CC_{\CC^n}[n])))\simeq j_{0}^{!}({}^p\psi_{f,\lambda}(\CC_{\CC^n}[n])/W_{k}({}^p\psi_{f,\lambda}(\CC_{\CC^n}[n])))$ has weights $> k$.
Therefore, taking $W_{k}$ of the sequence (\ref{con3ex2}),
we obtain
\begin{align}\label{a0}
W_{k}j_{0}^{-1}W_{k}({}^p\psi_{f,\lambda}(\CC_{\CC^n}[n]))\simeq W_{k}j_{0}^{-1}({}^p\psi_{f,\lambda}(\CC_{\CC^n}[n])).
\end{align}
On the other hand, since the functor $j_{0}^{-1}$ preserves the property that a complex of mixed Hodge modules has weights $\leq k$,
$j_{0}^{-1}W_{k}({}^p\psi_{f,\lambda}(\CC_{\CC^n}[n]))$ has weights $\leq k$.
Hence we have
\begin{align}\label{a1}
W_{k}j_{0}^{-1}W_{k}({}^p\psi_{f,\lambda}(\CC_{\CC^n}[n]))\simeq j_{0}^{-1}W_{k}({}^p\psi_{f,\lambda}(\CC_{\CC^n}[n])).
\end{align}
Combining the isomorphisms (\ref{a0}) and (\ref{a1}),
we get the desired isomorphism.
\end{proof}

\begin{proof}[Proof of Theorem~\ref{main2}]
Assume that $\lambda\notin R_{f}$.
We denote by $N$ the logarithm of the unipotent part of the monodromy automorphism of ${}^p\psi_{f,\lambda}(\CC_{\CC^n}[n])$,
and by $N_{0}$ its restriction to $0$, i.e. the logarithm operator of the unipotent part of $\Phi_{n-1}$ of $j_{0}^{-1}({}^p\psi_{f,\lambda}(\CC_{\CC^n}[n]))\simeq H^{n-1}(F_{f,0};\CC)_{\lambda}$.
Recall that for any $k\in\ZZ_{\geq 1}$
we have
\begin{align}\label{monofil}
N^{k}\colon \GR^{W}_{n-1+k}({}^p\psi_{f,\lambda}(\CC_{\CC^n}[n]))\simar \GR^{W}_{n-1-k}({}^p\psi_{f,\lambda}(\CC_{\CC^n}[n])).
\end{align}
Applying the functor $j_{0}^{-1}$ to the both sides of (\ref{monofil}),
by Proposition~\ref{compare1propgen} and Lemma~\ref{compare3} we obtain
\begin{align}\label{monofilmilAD}
N_{0}^{k}\colon \GR^{W}_{n-1+k}j^{-1}_{0}({}^p\psi_{f,\lambda}(\CC_{\CC^n}[n]))\simar \GR^{W}_{n-1-k}j^{-1}_{0}({}^p\psi_{f,\lambda}(\CC_{\CC^n}[n]))
\end{align}
that is
\[N_{0}^{k}\colon \GR^{W}_{n-1+k}H^{n-1}(F_{f,0};\CC)_{\lambda}\simar \GR^{W}_{n-1-k}H^{n-1}(F_{f,0};\CC)_{\lambda}\]
for any $k\in\ZZ_{\geq 1}$.
This implies that the weight filtration on $H^{n-1}(F_{f,0};\CC)_{\lambda}$ coincides with the monodromy filtration centered at $n-1$.
\end{proof}

\begin{rem}\label{newproof}
We can also prove Theorem~\ref{main2} in the following way.
First, we show the following general fact:
for a mixed Hodge module $\calM$ on $\CC^n$
if the natural morphism $H^kj_{0}^{!}\calM\to H^kj_{0}^{*}\calM$ is an isomorphism for any $k\in \ZZ$, 
then we have the following:
\begin{enumerate}
\item[(i)] For $k\neq 0$, we have $H^{k}j_{0}^{*}\calM=0$.
\item[(ii)] The natural morphism $(j_{0})_{*}H^0j_{0}^{*}\calM\to \calM$ is a monomorphism in $\mathrm{MHM}(\CC^n)$ and $(j_{0})_{*}H^0j_{0}^{*}\calM$ is a direct summand of $\calM$.
\end{enumerate}
Let $\calF$ be the underlying perverse sheaf of $\calM$.
We remark that the natural morphism $H^kj_{0}^{!}\calM\to H^kj_{0}^{*}\calM$ is an isomorphism for any $k\in \ZZ$
if and only if the natural morphism $j_{0}^{!}\calF\to j_{0}^{-1}\calF$ is an isomorphism since the functor $\mathrm{MHM}(\CC^n)\lr \Perv{\CC^n}$ is faithful.
Next, for $\lambda\notin R_{f}$ we consider the sub $\RR$-mixed Hodge module $\calM$ of $\psi_{f}^{H}(\RR^{H}_{\CC^n}[n])$ such that the complexification of the underlying perverse sheaf is $\calF={}^p\psi_{f,\lambda}(\CC_{\CC^n}[n])\oplus {}^p\psi_{f,\bar{\lambda}}(\CC_{\CC^n}[n])$.
Note that $\lambda\notin R_{f}$ implies $\bar{\lambda}\notin R_{f}$ by definition.
Then, we can apply the above fact to $\calM$ by Proposition~\ref{compare1prop},
and we thus obtain 
${}^p\calH^{k}(j_{0}^{-1}\calF)=0\ (k\neq 0)$ and
the filtration $(j_{0})_{*}(W_{\bullet}{}^p\calH^0j_{0}^{-1}\calF)(=W_{\bullet}(j_{0})_{*}{}^p\calH^{0}j_{0}^{-1}\calF)$ is a direct summand of $W_{\bullet}\calF$.
In this way, we can deduce the isomorphism (\ref{monofilmilAD}) from the isomorphism (\ref{monofil}), and we get Theorem~\ref{main2}.
\end{rem}

\begin{rem}
Let $f\in \CC\{x_{1},\dots,x_{n}\}$ be a convergent power series with $f(0)=0$.
In this case, we can also define the notions in Section~\ref{chapmil}:
the Milnor fiberation, the Milnor fiber $F_{f,0}$, the Milnor monodromies $\Phi_{j}\colon H^{j}(F_{f,0};\CC)\simar H^{j}(F_{f,0};\CC)$, the Newton polyhedron $\Gamma_{+}(f)$, the non-degeneracy at $0$ and the finite set $R_{f}\subset \CC$, similarly to the case where $f$ is a polynomial.
We can also consider a (analytic) mixed Hodge module $\psi_{f}(\QQ^{H}_{\CC^n}[n])$ and a mixed Hodge structure $H^{k}j_{0}^{*}(\psi_{f}(\QQ^{H}_{\CC^n}[n]))$, whose underlying vector space is $H^{k+n-1}(F_{f,0};\QQ)$ for $k\in \ZZ$.
Assume that $f$ is non-degenerate at $0$.
Then, even in this setting we can prove Proposition~\ref{compare1prop}, Theorem~\ref{main1}, Proposition~\ref{compare1propgen}, Lemma~\ref{compare3} and Theorem~\ref{main2} in the same way
(We remark that we do not have the $6$-operations between the derived categories of (analytic) mixed Hodge modules on analytic spaces in general. Therefore we have to be careful to use $\mathrm{D^bMHM}(\CC^n)$ in the proof of Lemma~\ref{compare3}. Nevertheless, in this setting, we have the functors $j_{0}^{*}$ and $j_{0}^{!}$ as in the case of the derived categories of algebraic mixed Hodge modules (see section 2.30 of \cite{MHM}). Thus, the same proofs work even in the case where $f$ is a covergent power series.).\\
The proof in Remark~\ref{newproof} works even in the analytic setting.
Therefore, also in this way, we obtain Theorem~\ref{main1} and \ref{main2} for a convergent power series $f$.
\end{rem}

\section{Applications}

In this section,
we apply Theorems~\ref{main1} and \ref{main2}
to compute the Jordan normal forms of the Milnor monodromies and
the Hodge spectra.
Let $f(x)\in\nPoly$ be a polynomial such that $f(0)=0$.
Assume that it is non-degenerate at $0$.
Let $P$ be the convex hull of $\Ntbd\cup\{0\}$.
Since our formula below will become trivial in the case when the dimension of $P$ less than $n$,
in what follows we assume that the dimension of $P$ is equal to $n$.
Moreover, since the case where $f$ is convenient was already treated by Matsui-Takeuchi~\cite{MT} and M. Saito~\cite{Mexp},
we assume that $f$ is not convenient in this section.
Then $R_{f}$ is not empty and contains $1\in\CC$. 
For $\lambda\notin R_{f}$, by Proposition~\ref{propsymEpoly}, we have 
\[u^2E_{\lambda}(F_{f,0};u,u)
=(-1)^{n-1}l^{*}_{\lambda}(P,\nu;u,u),
\]
where $\nu$ is the piecewise linear function on $P$ defined in Section~\ref{chaphreal}.
Recall that $\calS_{\nu}$ is the polyhedral subdivision of $P$ defined by $\nu$.
By the definition of the $h^*$-polynomial, for $\lambda\notin R_{f}$ we have
\[l^{*}_{\lambda}(P,\nu;u,u)=
\sum_{F\prec\Ntph\sumcom{admissible}}u^{\dim{\Delta_{F}}+1}
l^{*}_{\lambda}(\Delta_{F},\nu;1)\cdot l_{P}(\calS_{\nu},\Delta_{F};u^2),\]
where in the sum $\Sigma$ the face $F$ ranges through the compact admissible ones of $\Ntph$.
The polynomial $l_{P}(\calS_{\nu},\Delta_{F};t)$ is symmetric and unimodal centered at $(n-\dim{F}-1)/2$, i.e. if $a_{i}\in \ZZ$ is the coefficient of $t^i$ in $l_{P}(\calS_{\nu},\Delta_{F};t)$ we have $a_{i}=a_{n-\dim{F}-1-i}$ and $a_{i}\leq a_{j}$ for $0\leq i\leq j\leq (n-\dim{F}-1)/2$.
Therefore, it can be expressed in the form
\[l_{P}(\calS_{\nu},\Delta_{F};t)=\sum_{i=0}^{\lfloor (n-1-\dim{F})/2\rfloor}\tl{l}_{F,i}(t^i+t^{i+1}+\dots+t^{n-1-\dim{F}-i}),\]
for some non-negative integers $\tl{l}_{{F},i}\in\ZZ_{\geq 0}$.
We set
\[\tl{l}_{P}(\calS_{\nu},\Delta_{F},t):=\sum_{i=0}^{\lfloor (n-1-\dim{F})/2\rfloor}\tl{l}_{F,i}t^i.\]
For $k\in\ZZ_{\geq 0}$ and $\lambda\in\CC$ we denote by $J_{k,\lambda}$ the number of the Jordan blocks in $\Phi_{n-1}$ with size $k$ for the eigenvalue $\lambda$.
Then by Theorems~\ref{main1} and \ref{main2}
we obtain the following formula for them.

\begin{cor}\label{jordan}
In the situation as above, 
for any $\lambda\notin R_{f}$ we have
\begin{align*}
\sum_{0\leq k\leq n-1}J_{n-k,\lambda}u^{k+2}
=\sum_{\substack{F\prec\Ntph\sumcom{admissible}}}
u^{\dim{\Delta_{F}}+1}l^{*}_{\lambda}(\Delta_{F},\nu;1)\cdot\tl{l}_{P}(\calS_{\nu},\Delta_{F};u^2),
\end{align*}
where in the sum $\Sigma$ of the right hand side the face $F$ ranges through the admissible ones of $\Ntph$.
\end{cor}

Finally, we introduce our formula for the Hodge spectrum of $f$ at the origin $0$.
\begin{defi}
\begin{enumerate}
\item
We define a Puiseux polynomial $\mathrm{sp}_{f,0}(t)$ with coefficients in $\ZZ$ by
\[\mathrm{sp}_{f,0}(t)=(-1)^{n-1}\sum_{\alpha\in\QQ\cap[0,n]}\left\{\sum_{j\in\ZZ}(-1)^{j}\dim{\GR^{\lfloor{n-\alpha}\rfloor}_{F}\tl{H^j}(F_{f,0};\CC)_{\exp(-2\pi\sqrt{-1}\alpha)}} \right\}t^{\alpha},\]
where $\GR^{\lfloor{n-\alpha}\rfloor}_{F}\tl{H^j}(F_{f,0};\CC)_{\exp(-2\pi\sqrt{-1}\alpha)}$
is the graded piece with respect to the Hodge filtration of the mixed Hodge structure of $\tl{H}^j(F_{f,0};\QQ)$.
We call it the \textit{Hodge spectrum of $f$ at $0$}.
\item For $\beta\in (0,1)\cap \QQ$ we set $\lambda=\exp(2\pi\sqrt{-1}\beta)$ and we define a Puiseux polynomial $\Sp_{f,0}^{\lambda}(t)$ by
\[\Sp_{f,0}^{\lambda}(t)=(-1)^{n-1}\sum_{i=0}^{n-1}\left\{\sum_{j\in\ZZ}(-1)^{j}\dim{\GR^{\lfloor{n-\beta-i}\rfloor}_{F}\tl{H^j}(F_{f,0};\CC)_{\lambda^{-1}}}\right\}
t^{\beta+i}.\]
\end{enumerate}
\end{defi}
Since $f$ is non-degenerate at $0$,
by setting $v=1$ in Corollary~\ref{epolycomp1} we can express $\Sp_{f,0}(t)$ and $\Sp_{f,0}^{\lambda}(t)$ in terms of $\Gamma_{f}$.
Moreover, 
if $\lambda=\exp(2\pi\sqrt{-1}\beta)$ is not in $R_{f}$,
$\Sp_{f,0}^{\lambda}(t)$ can be rewritten much more simply as follows.
For a compact face $F$ of $\Ntph$, we define a cone $\mathrm{Cone}(F)\subset \RR^n$ by $\mathrm{Cone}(F):= \RR_{\geq 0}F$
and the linear function $h_{F}$ on $\mathrm{Cone}(F)$ which takes the value $0$ at the origin $0\in\RR^n$ and the value $1$ on $F$.
Moreover, for $\beta\in(0,1)\cap \QQ$
we define a Puiseux polynomial $P_{F,\beta}(t)$ by
\[P_{F,\beta}(t):=\sum_{i=0}^{+\infty}\#\{v\in \mathrm{Cone}(F)\cap \ZZ^n_{\geq 0} \ |\ h_{F}(v)=\beta+i\}t^{\beta+i}.\]
Then we obtain the following formula,
which generalizes the one for $\Sp_{f,0}(t)$ in the case where $0\in V$ is an isolated singular point obtained by M. Saito~\cite{Mexp}.
For the corresponding result for the monodromies at infinity, see Theorem~{5.16} of Matsui-Takeuchi~\cite{MTinf}.

\begin{cor}\label{spfor}
In the situation as above, assume moreover that $\lambda$ is not in $R_{f}$.
Then we have
\[\Sp_{f,0}^{\lambda}(t)=\sum_{F\prec\Ntph\sumcom{admissible}}(-1)^{n-1-\dim{F}}(1-t)^{s_{F}}P_{F,\beta}(t),\]
where in the sum $\Sigma$ the face $F$ ranges through the admissible ones of $\Ntph$ and $s_{F}\in\ZZ_{\geq 1}$ is the integer defined in Section~\ref{subsecmot}. 
\end{cor}

\begin{proof}
By Theorem~\ref{main1}, for $\beta\in (0,1)\cap\QQ$ such that $\lambda=\exp(2\pi\sqrt{-1}\beta)\notin R_{f}$ we have
the concentration
\[\tl{H^j}(F_{f,0};\CC)_{\lambda}\simeq 0 \qquad (j\neq n-1).\]
Moreover, for $0\leq i\leq n-1$
we have
\begin{align*}
\dim{\GR^{\lfloor{n-i-\beta}\rfloor}_{F}\tl{H}^{n-1}(F_{f,0};\CC)_{\lambda^{-1}}}
&=\sum_{k\in\ZZ}\dim{\GR^{\lfloor{n-i-\beta}\rfloor}_{F}\GR_{k}^{W}\tl{H}^{n-1}(F_{f,0};\CC)_{\lambda^{-1}}}\\
&=\sum_{k\in\ZZ}\dim\GR^{n-1-\lfloor{n-i-\beta}\rfloor}_{F}\GR_{2(n-1)-k-\lfloor{n-\alpha}\rfloor}^{W}\tl{H}^{n-1}(F_{f,0};\CC)_{\lambda}\\
&=\dim\GR^{i}_{F}\tl{H}^{n-1}(F_{f,0};\CC)_{\lambda},
\end{align*}
where in the second (resp. third) equality we used Theorem~\ref{propsymEpoly} and Proposition~\ref{hformula}~(ii) (resp. $n-1-\lfloor{n-i-\beta}\rfloor=i$).
Then we have
\begin{align*}
\Sp_{f,0}^{\lambda}(t)
&=\sum_{i=0}^{n-1}\dim\GR^{\lfloor n-i-\beta\rfloor}_{F}H^{n-1}(F_{f,0};\CC)_{\lambda^{-1}}t^{\beta+i}\\
&=\sum_{i=0}^{n-1}\dim\GR^{i}_{F}H^{n-1}(F_{f,0};\CC)_{\lambda}t^{\beta+i}\\
&=(-1)^{n-1}E_{\lambda}(F_{f,0};t,1)t^{\beta}\\
&=\frac{1}{t}l^{*}_{\lambda}(P,\nu;t)t^{\beta} \qquad (\mbox{by Propositions~\ref{propsymEpoly} and \ref{hformula}~(i)}),
\end{align*}
Furthermore, 
we have
\begin{align*}
&\qquad t^{\beta-1}l^{*}_{\lambda}(P,\nu;t)\\
&=t^{\beta-1}\sum_{{F\prec \Ntph\sumcom{admissible}}}\left\{(-1)^{n-\dim\sigma(\Delta_{F})}(t-1)^{\dim\sigma(\Delta_{F})-\dim\Delta_{F}}h^*_{\lambda}(\Delta_{F},\nu;t)\right\}\\
&=t^{\beta-1}\sum_{{F\prec \Ntph\sumcom{admissible}}}\left\{(-1)^{n-\dim\sigma(\Delta_{F})}(t-1)^{\dim\sigma(\Delta_{F})-\dim\Delta_{F}}(1-t)^{\dim{\Delta_{F}+1}}
\sum_{m\geq 0}f_{\lambda}(P,\nu;m)t^m\right\}\\
&=t^{\beta-1}\sum_{F\prec \Ntph\sumcom{admissible}}\left\{(-1)^{n-1-\dim{F}}(1-t)^{\dim{\sigma({\Delta_{F}})}+1}
\sum_{m\geq 0}f_{\lambda}(P,\nu;m)t^m\right\}\\
&=t^{\beta-1}\sum_{{F\prec \Ntph\sumcom{admissible}}}\left\{(-1)^{n-1-\dim{F}}(1-t)^{\dim{\sigma(\Delta_{F})}}
\sum_{m\geq 1}\lfloor(f_{\lambda}(P,\nu;m)-f_{\lambda}(P,\nu;m-1)\rfloor)t^{m}\right\}\\
&=\sum_{F\prec \Ntph\sumcom{admissible}}(-1)^{n-1-\dim{F}}(1-t)^{\dim{\sigma({\Delta_{F}})}}P_{F,\beta}(t),
\end{align*}
where in the sums $\Sigma$ the faces $F$ range through the admissible ones of $\Ntph$.
Since for an admissible face $F\prec\Ntph$ we have $\dim{\sigma(\Delta_{F})}=s_{F}$, this completes the proof.
\end{proof}

\bibliography{reference}

\end{document}